\documentclass[a4paper,12pt]{article}

\usepackage{authblk}
\usepackage{amsthm,latexsym,amsmath,amssymb,amsfonts,euscript,dsfont}
\usepackage{epsfig}
\usepackage{mathtools}
\usepackage{mathrsfs}
\usepackage{relsize} 
\usepackage{comment}
\usepackage{ushort}
\usepackage{color}
\usepackage{url}
\usepackage{hyperref}

\newtheorem{theorem}{{\bf{\small T}{\scriptsize HEOREM}}}[section]
\newtheorem{corollary}{{\bf{\small C}{\scriptsize OROLLARY}}}[section]
\newtheorem{proposition}{{\bf{\small P}{\scriptsize ROPOSITION}}}[section]
\newtheorem{lemma}{{\bf{\small L}{\scriptsize EMMA}}}[section]
\newtheorem{remark}{{\bf{\small R}{\scriptsize EMARK}}}[section]

\newtheorem{definition}{{\bf{\small D}{\scriptsize EFINITION}}}[section]

\renewenvironment{proof}[1]
{\noindent{{\bf{\small{ P}{\scriptsize ROOF}}}.}\hspace{0.1cm} #1} {$\;\qed$\newline}

\def\R{\mathds R}
\def\N{\mathds N}
\def\Z{\mathds Z}
\newcommand{\Zd}{\mathds Z^d}
\def\un{\mathds{1}}
\def\E{\mathds E}
\def\pee{\mathds P}
\def\Oun{\mathcal{O}(1)}

\def\bgamma{\boldsymbol\gamma}

\newcommand{\dd}{\mathop{}\!\mathrm{d}}

\DeclareMathOperator{\sign}{sign}

\font\gfont=cmmi10 scaled \magstep{1.5}    
\font\ggfont=cmmi10 scaled \magstep{2}

\newcommand{\gdelta}{\hbox{\gfont \char14}}
\newcommand{\ggdelta}{\hbox{\ggfont \char14}}

\newcommand{\dist}{d}
\newcommand{\distk}{d_{{\scriptscriptstyle K}}}
\newcommand{\db}{\bar{d}}

\newcommand{\geta}{\hbox{\gfont \char17}}

\newcommand{\gmu}{\hbox{\gfont \char22}}
\newcommand{\gnu}{\hbox{\gfont \char23}}

\newcommand{\gsigma}{\hbox{\gfont \char27}}

\makeatletter
\newcommand{\opnorm}{\@ifstar\@opnorms\@opnorm}
\newcommand{\@opnorms}[1]{%
  \left|\mkern-1.5mu\left|\mkern-1.5mu\left|
   #1
  \right|\mkern-1.5mu\right|\mkern-1.5mu\right|
}
\newcommand{\@opnorm}[2][]{%
  \mathopen{#1|\mkern-1.5mu#1|\mkern-1.5mu#1|}
  #2
  \mathclose{#1|\mkern-1.5mu#1|\mkern-1.5mu#1|}
}
\makeatother

\newcommand{\gcb}[1]{\mathrm{GCB}\!\left(#1\right)}
\newcommand{\mcb}[2][]{\mathrm{MCB}\!\left(#1#2\right)}

\begin{document}

\title{On concentration inequalities\\ and their applications\\ for Gibbs measures in lattice systems}

\author[1]{J.-R. Chazottes
\thanks{Email: \texttt{chazottes@cpht.polytechnique.fr}}}

\author[1]{P. Collet
\thanks{Email: \texttt{collet@cpht.polytechnique.fr}}}

\author[2]{F. Redig
\thanks{Email: \texttt{F.H.J.Redig@tudelft.nl}}}

\affil[1]{Centre de Physique Th\'eorique, CNRS UMR 7644, F-91128 Palaiseau Cedex (France)}
\affil[2]{Delft Institute of Applied Mathematics, Technische Universiteit Delft, Nederland}

\date{Dated: \today}

\maketitle

\begin{abstract}
We consider Gibbs measures on the configuration space $S^{\Zd}$,  where mostly $d\geq 2$ and $S$
is a finite set. We start by a short review on concentration inequalities for Gibbs measures.
In the Dobrushin uniqueness regime, we have a Gaussian concentration bound, whereas in
the Ising model (and related models) at sufficiently low temperature, we control all moments and have a stretched-exponential concentration bound. 
We then give several applications of these inequalities whereby we obtain various new results.
Amongst these applications, we get bounds on the speed of convergence of the empirical measure in the sense of Kantorovich distance, fluctuation bounds in the Shannon-McMillan-Breiman theorem, fluctuation bounds for the first occurrence of a pattern, as well as almost-sure central limit theorems.

\smallskip

\noindent {\footnotesize{\bf Keywords and phrases:} Gaussian concentration bound, moment concentration bound,  low-temperature Ising model, Dobrushin uniqueness, $\db$-distance, empirical measure, relative entropy, Kantorovich distance, almost-sure central limit theorem.}
\end{abstract}

\maketitle


\tableofcontents

\section{\textbf{Introduction}}

Concentration inequalities play by now an important role in probability theory and statistics, as well as in various areas such as geometry, functional analysis, discrete mathematics \cite{blm,dp,ledoux}. Remarkably, the scope
of these inequalities ranges from the more abstract to the explicit analysis of given models.
With a view towards our setting, the elementary manifestation of the concentration of measure phenomenon can be formulated as follows.
Consider independent random variables $\{\omega_x, x \in C_n\}$ taking the values $\pm 1$ with equal probability and indexed by the sites of a large but finite discrete cube $C_n$ of ``side length'' $2n+1$ in $\Zd$. The partial sum $\sum_{x\in C_n} \omega_x$ has expectation zero. Of course, this sum varies in an interval of size $\mathcal{O}(n^d)$. But, in fact, it sharply concentrates with very high probability in a much narrower range, namely in an interval of size
$\mathcal{O}(n^{d/2})$. This statement is quantified by the following ``Gaussian bound'' or Hoeffding inequality (see \cite{blm}): 
\[
\pee\left\{ \left|\sum_{x\in C_n} \omega_x\right|\geq u\, (2n+1)^{d/2} \right\}\leq
2\, \exp\left(-\frac{u^2}{2}\right)
\]
for all $n\geq 1$ and for all $u>0$. This is a finite-volume quantitative version of the strong law of large numbers, giving the correct scale as in the central limit theorem.
This phenomenon is not tied to linear combinations of the $\omega_x$'s, like the above sum, but in fact holds for a broad class of nonlinear functions $F$ of the $\omega_x$'s. Thus, we can get tight bounds for the probability that a complicated
or implicitly defined function of the $\omega_x$'s deviates from its expected value. Let us stress that 
concentration inequalities are valid in every finite volume, and not just asymptotically. 

Now, what happens if the $\omega_x$'s are no longer independent? One can expect to still have a Gaussian bound of the same flavour as above provided correlations are weak enough amongst the
$\omega_x$'s (see {\em e.g.} \cite{jrf} about Markov chains, and \cite{kr} for a survey focused on the martingale method). 
In the present paper,  we are interested in Gibbs measures on a configuration space of the form
$\Omega=S^{\Zd}$ where $S$ is a finite set. In the above elementary example, we have $S=\{-1,1\}$ (spins) and the previously considered product measure can be thought as a Gibbs measure at infinite temperature. The first work in this setting is \cite{Kul} in which it was proved that a Gaussian concentration bound holds in  Dobrushin's uniqueness regime (see below for a precise statement). The constant appearing in the bound is directly related to the ``Dobrushin contraction coefficient''. For instance, any finite-range potential at sufficiently high temperature satisfies Dobrushin's condition, like the Ising model. One of the main motivations of \cite{CCKR} was to figure out what happens for the Ising model at {\em low temperature}. One cannot expect that a Gaussian concentration bound holds (see details below), and it was proved in \cite{CCKR} that a stretched-exponential decay of the form $\exp(-c u^\varrho)$ holds, where $0<\varrho <1$ depends on the temperature. Notice that we deal with $d\geq 2$. For $d=1$, the situation is as follows. Finite-range potentials give rise to finite-state Markov chains and thus one has a Gaussian concentration bound. For potentials which are summable in certains sense, one has also a Gausian concentration bound, but the known results are formulated in terms of chains of infinite order (or $g$-measures) rather than Gibbs measures, see \cite{st}.
For long-range potentials, like Dyson models, nothing is known regarding concentration bounds. In that context, let us mention that $g$-measures can be different from Gibbs measures, see \cite{bissacot} and references therein.

The purpose of the present work is to apply these concentration bounds to various types of functions $F$ of the  $\omega_x$'s, both in Dobrushin's uniqueness regime and in the Ising model at sufficiently low temperature. 
For example, we obtain quantitative estimates for the speed of convergence of the empirical measure
to the underlying Gibbs measure in Kantorovich distance. In the Ising model, this speed depends in particular on the temperature regime. Here the estimation of the expected distance raises an extra problem which requires to adapt methods used to estimate suprema of empirical processes. The problem comes from the fact that our configuration
space is topologically a Cantor set.
Another application concerns ``fattening'' finite configurations in the sense of Hamming distance: take,
{\em e.g.}, $S=\{-1,1\}$ and consider the set $\Omega_n=\{\eta_x: x\in C_n\}$.
Now, take a subset $\EuScript{B}_n\subset \Omega_n$ of, say, measure $1/2$, and look at the set
$\EuScript{B}_{n,\epsilon}$ of all configurations in $\Omega_n$ obtained from those in $\EuScript{B}_n$ by flipping, say,
$\epsilon=5\%$ of the spins. It turns out that, for large but finite $n$, the set $\EuScript{B}_{n,\epsilon}$ has probability very close to $1$. 
Besides fluctuation bounds, we also obtain an almost-sure central limit theorem, thereby showing how concentration inequalities can also lead to substantial reinforcements of weak limit theorems in great generality.
 
Concentration inequalities may look weaker than a ``large deviation principle'' \cite{dz}. On one hand, this is true because getting a large deviation principle means that one gets a rate function which gives the correct asymptotic exponential decay to zero of the probability that, {\em e.g.}, $(2n+1)^{-d}\sum_{x\in C_n} \omega_x$ deviates from its expectation (the magnetization of the system). But, on the other hand, it is hopeless to get a large deviation principle for functions of the $\omega_x$'s which do not have some (approximate) additivity property. This rules out many interesting functions of the $\omega_x$'s.
Besides, even in the situation when concentration inequalities and large deviation principles coexist, the former provides simple and useful bounds which are valid in every finite volume.

We also emphasize that concentration inequalities provide upper-bounds which are ``permutation invariant''. In particular, for averages of the form $|\Lambda|^{-1}\sum_{x\in\Lambda} f(T_x \omega)$ one obtains bounds in which the dependence on $\Lambda$ is only through its cardinality, and thus insensitve to its shape.
In the case of the Gaussian concentration bound, one obtains an upper bound for
the logarithm of the exponential moment of $\sum_{x\in\Lambda} f(T_x \omega)$ which is of the order
$|\Lambda|$.
This provides an order of growth as would be provided by large deviation theory in contexts where the latter
is not necessarily available. Indeed, in order to have a large deviation principle, it is necessary that the sets
$\Lambda$ grow as a van Hove sequence, see {\em e.g.} \cite{EKW}. 
An illustrative example is when $\Lambda$ is a subset of $\Z^d$ which is contained in a hyperplane of lower dimension ({\em e.g.}, a subset of one of the coordinate planes). Indeed, there is a priori no large deviation principle available for projections of Gibbs measures on lower dimensional sets (they might fail to satisfy the variational principle), whereas concentration bounds are still possible.

Before giving the outline of this paper, let us mention the papers \cite{dedecker}, \cite{chatterjee1,chatterjee-dey}, and \cite{marton1,marton2}, which deal with concentration inequalities for spin models from statistical mechanics. In \cite{dedecker}, the author establishes, among other things, a Gaussian concentration bound for partial sums of a random field satisfying a ``weak mixing'' condition. This includes the Ising model above its critical temperature. In \cite{chatterjee1,chatterjee-dey}, the authors obtain concentration inequalities for mean-field models, like the Curie-Weiss model. These results follow from a method introduced by Chatterjee in \cite{chatterjee1} (a version of Stein's method).

The rest of our paper is organized as follows. After some generalities on concentration bounds given in Section \ref{sec:genconc} and tailored for our needs, we gather a number of facts on Gibbs measures which we will use in our applications (Section \ref{sec:Gibbs}). We then review the known concentration properties of Gibbs measures, {\em i.e.},
the Gaussian concentration bound which is valid in Dobrushin's uniqueness regime (Section \ref{sec:gcbgauss}), and the moment inequalities, as well as a stretched-exponential concentration bound, which hold for the Ising model at sufficiently low temperature (Section \ref{sec:lowtempconc}). 
Then we derive various applications of the concentration bounds in Sections \ref{sec:app1}-\ref{sec:ASCLT}.

\section{\textbf{Setting}}

\subsection{Configurations and shift action}

We work with the configuration space $\Omega=S^{\Zd}$, where $S$ is a finite set, and $d$ an integer greater than or equal to $2$.
We endow $\Omega$ with the product topology that is generated by cylinder sets.
We denote by $\mathfrak{B}$ the Borel $\sigma$-algebra which coincides with the $\sigma$-algebra generated by these sets.

An element $x$ of $\Zd$ (hereby called a site) can be written as a vector $(x_1,\ldots,x_d)$ in the canonical base of the lattice $\Zd$.
Let $\|x\|_\infty=\max_{1\leq i\leq d} |x_i|$, and denote by $\|x\|_1$ the Manhattan norm,
that is, $\|x\|_1=|x_1|+\cdots + |x_d|$. More generally, given an integer
$p\geq 1$, let $\|x\|_p=(|x_1|^p+\cdots+|x_d|^p)^{1/p}$. If $\Lambda$ is a finite subset of $\Zd$, denote by $\textup{diam}(\Lambda)=\max\{\|x\|_\infty : x\in\Lambda\}$ its diameter, and by $|\Lambda|$ its cardinality. The collection of finite subsets of $\Zd$ will be denoted by $\mathcal{P}$.

We consider the following distance on $\Omega$ : for $\omega,\omega'\in\Omega$, let
\begin{equation}\label{def:dist}
\dist(\omega,\omega')=2^{-k}\quad\textup{where}\; k=\min\{\|x\|_\infty : \omega_x\neq \omega'_x\}.
\end{equation}
This distance induces the product topology, and one can prove that $(\Omega,d)$ is a compact metric space.
Note that $\Omega$ is a Cantor set, so it is totally disconnected.

For $\Lambda\subset\Zd$, we denote by $\Omega_\Lambda$ the projection of $\Omega$
onto $S^\Lambda$. Accordingly, an element of $\Omega_\Lambda$ is denoted by $\omega_\Lambda$ and
is viewed as a configuration $\omega\in\Omega$ restricted to $\Lambda$. Another useful notation is the following. For $\gsigma,\geta\in \Omega$ we denote by $\gsigma_\Lambda\geta_{\Lambda^c}$ the configuration which agrees with $\gsigma$ on $\Lambda$ and with $\geta$ on $\Lambda^c$. Finally, we denote by $\mathfrak{B}_\Lambda$ the $\sigma$-algebra generated by the coordinate maps $f_x:\omega\mapsto \omega_x$, $x\in\Lambda$.

Subsets of particular interest are cubes centered about the origin of $\Zd$: for every $n\in\N$, define
\[
C_n=\big\{x\in\Zd : -n\leq x_i\leq n,\,i=1,2,\ldots,d\big\}.
\]
For $\omega\in \Omega$ and $n\in\N$, define the cylinder set
\[
\EuScript{C}_n(\omega)=\{\eta\in\Omega : \eta_{C_n}=\omega_{C_n}\}.
\]
We simply write $\Omega_n$ for $\Omega_{C_n}$ which is the set of partial configurations supported on
$C_n$.

Finally, the shift action $(T_x,x\in\Zd)$ is defined as usual: for each $x\in\Zd$, $T_x:\Omega\to \Omega$
and $(T_x\omega)_y=\omega_{y-x}$, for all $y\in\Zd$. This corresponds to translating $\omega$ forward by $x$.

\subsection{Functions}

Let $F:\Omega\to\R$ be a continuous function and $x\in\Zd$. We denote by
\[
\gdelta_x (F) = \sup\big\{ |F(\omega) - F(\omega')| : \omega,\omega'\in \Omega\;\textup{differ only at site}\;x\big\}
\]
the oscillation of $F$ at $x$. It is a natural object because, given a finite subset $\Lambda\subset \Zd$ and
two configurations $\omega,\eta\in\Omega$ such that  $\omega_{\Lambda^c}=\eta_{\Lambda^c}$, one has
\[
|F(\omega)-F(\eta)| \leq \sum_{x\in\Lambda} \delta_x(F).
\]
We shall say that  $F:\Omega\to\R$ is a {\em local function} if there exists a finite subset $\Lambda_F$
of $\Zd$ (the dependence set of $F$) such that for all $\omega,\widetilde\omega,\widehat\omega$,
$F(\omega_{\Lambda_F}\widetilde\omega_{\Lambda^c_F}) = F(\omega_{\Lambda_F}\widehat\omega_{\Lambda^c_F})$. Equivalently, $\delta_x(F)=0$ for all $x\notin \Lambda_F$.
It is understood that $\Lambda_F$ is the smallest such set.
When $\Lambda_F=C_n$ for some $n$, $F$ is said to be ``cylindrical''.

Let $C^0(\Omega)$ be the Banach space of continuous functions $F:\Omega\to\R$
equipped with supremum norm $\|F\|_\infty=\sup_{\omega\in\Omega} |F(\omega)|$.
Every local function is continuous and the uniform closure of the set of all local functions is $C^0(\Omega)$.
Given $F$, we write $\ushort{\gdelta}(F)$ for the infinite array $(\gdelta_x(F), x\in\Zd)$.
For every $p\in\N$,  we introduce the semi-norm
\[
\|\ushort{\gdelta} (F)\|_p:=
\|\ushort{\gdelta} (F)\|_{\ell^p(\Zd)}= \Big(\sum_{x\in\Zd} (\gdelta_x(F))^p\Big)^{1/p}.
\]
Finally, we define the following spaces of functions:
\begin{equation}\label{def-Delta-space}
\Delta_p(\Omega)=C^0(\Omega)\cap \left\{ F:\Omega\to\R: F\in  \|\ushort{\gdelta} (F)\|_p<\infty\right\},\,
p\in\N.
\end{equation}
Each of these spaces obviously contains local functions, and $\Delta_p(\Omega)\subset \Delta_q(\Omega)$
if $1\leq p<q\leq+\infty$. Notice that the space of functions such that $\|\ushort{\gdelta} (F)\|_p<\infty$ for
a given $p\in\N$ is neither contained in nor contains $C^0(\Omega)$.

Define the oscillation of a function $F:\Omega\to\R$ as
\[
\delta(F)=\sup F -\inf F = \sup_{\omega,\omega'\in\Omega} |F(\omega)-F(\omega')|. 
\]
If $F\in C^0(\Omega)$, one has
\[
\|\ushort{\gdelta} (F)\|_1=\sum_{x\in\Zd} \delta_x(F)\geq \delta(F).
\]
For $p\in\N$, the semi-norm $\|\ushort{\gdelta} (\cdot)\|_p$ becomes a norm if one considers the quotient space where two functions in $\Delta_p(\Omega)$ are declared to be equivalent if their difference is a constant function. Moreover, this quotient space equipped with the norm $\|\ushort{\gdelta} (\cdot)\|_p$ is a Banach space.

\section{\textbf{Concentration bounds for random fields: abstract definitions and consequences}}\label{sec:genconc}

We state some abstract definitions and their general consequences that we will use repeatedly in the
sequel.

\subsection{Gaussian concentration bound}

\begin{definition}
\leavevmode\\
Let $\nu$ be a probability measure on $(\Omega,\mathfrak{B})$. 
We say that it satisfies the Gaussian concentration bound with constant
$D=D(\nu)>0$ (abbreviated $\gcb{D}$) if, for all functions $F\in\Delta_2(\Omega)$, we have 
\begin{equation}\label{eq:gemb}
\E_\nu \big[\exp\left(F - \E_\nu [F]\right)\big] \leq \exp\left(D\|\ushort{\delta} (F) \|_2^2\right). 
\end{equation}
\end{definition}

A key point in this definition is that $D$ is independent of $F$. Inequality \eqref{eq:gemb} easily implies Gaussian concentration inequalities that we gather in the following proposition in a convenient form for later use.

\begin{proposition}\label{prop-exp-dev}
\leavevmode\\
If a probability measure $\nu$ on $(\Omega,\mathfrak{B})$ satisfies $\gcb{D}$ then,
for all functions $F\in\Delta_2(\Omega)$ and for all $u>0$, one has
\begin{align}
\label{eq:expcon1}
& \gnu\left\{\omega\in \Omega : F(\omega) - \E_\nu [F] \geq u \right\} \leq
\exp\left(-\frac{u^2}{4D\|\ushort{\delta}(F) \|_2^2}\right)\, ,\\
\label{eq:expcon2}
& \gnu\left\{\omega\in \Omega : |F(\omega) - \E_\nu[F]| \geq u \right\} \leq
2\ \exp\left(-\frac{u^2}{4D\|\ushort{\delta} (F) \|_2^2}\right). 
\end{align}
\end{proposition}

\begin{proof}
If $F\in\Delta_2(\Omega)$, then $\lambda F\in\Delta_2(\Omega)$ for any $\lambda\in\R_+$.
We apply Markov's inequality and \eqref{eq:gemb} to get
\begin{align*}
\gnu\left\{\omega\in \Omega : F(\omega) - \E_\nu [F] \geq u \right\} 
&\leq  \exp\left(-\lambda u\right)\ \E_\nu\big[\exp\left(\lambda(F - \E_\nu [F])\right) \big]\\
& \leq \exp\left(-\lambda u+D\|\ushort{\delta} (F) \|_2^2\, \lambda^2\right).
\end{align*}
We now optimize over $\lambda$ to get \eqref{eq:expcon1}.
Applying this inequality to $-F$ gives the same inequality if `$\geq u$' is replaced by
`$\leq -u$', whence
\begin{align*}
\MoveEqLeft \gnu\left\{\omega\in \Omega : |F(\omega) - \E_\nu [F]| \geq u \right\}\\ 
&\leq \gnu\left\{\omega\in \Omega : F(\omega) - \E_\nu [F] \geq u \right\}\!+
\gnu\left\{\omega\in \Omega : F(\omega) - \E_\nu [F] \leq -u \right\} \\
& \leq 2\ \exp\left(-\frac{u^2}{4D\|\ushort{\delta} (F) \|_2^2}\right),
\end{align*}
which is \eqref{eq:expcon2}.
\end{proof}

\subsection{Moment concentration bounds}

\begin{definition}\label{def:moments}
\leavevmode\\
Given $p\in\N$, we say that a probability measure $\nu$ on $(\Omega,\mathfrak{B})$
satisfies the moment concentration bound of order $2p$ with constant $C_{2p}=C_{2p}(\nu)>0$
(abbreviated $\mcb{2p,C_{2p}}$) if, for all functions $F\in\Delta_2(\Omega)$, we have
\begin{equation}\label{eq:momentb}
\E_\nu\left[ (F-\E_\nu [F])^{2p}\right]\leq C_{2p}\, \|\ushort{\gdelta} (F)\|^{2p}_2\,.
\end{equation}
\end{definition}

Again, as for the Gaussian concentration bound, the point is that the involved constant, namely $C_{2p}$,
is required to be independent of $F$.
An application of Markov's inequality immediately gives the following polynomial concentration inequality:
\begin{equation}\label{eq:momentdev}
\gnu\left\{\omega\in \Omega :  |F(\omega) -\E_\nu [F]| >u\right\}
\leq \frac{C_{2p}\, \| \ushort{\gdelta}(F)\|_2^{2p}}{u^{2p}}
\end{equation}
for all $u>0$.

\subsection{Gaussian tails and growth of moments}\label{rem:moments-vs-gb}

Let $Z$ be a real-valued random variable with $\E[Z]=0$. If for some positive constant $K$
\[
\E[Z^{2p}]\leq p! K^p,\;\forall p\in\N, 
\]
then $\E[e^{\lambda Z}]\leq e^{2K \lambda^2}$ for all $\lambda\in\R$.
Applied to $Z=F-\E_\nu[F]$ for a probability measure $\nu$
satisfying $\gcb{D}$ for all $p\in\N$, this gives a road to establishing that $Z$ 
satisfies a Gaussian concentration bound.

Conversely, if there exists a constant $K>0$ such that for all $u>0$
\[
\max\{\pee(Z\geq u),\pee(-Z\geq u)\}\leq \exp\left(-\frac{u^2}{2K}\right),
\] 
then for every integer $p\geq 1$,
\[
\E[Z^{2p}]\leq p!\, (4K)^p.
\]
Applied to $Z=F-\E_\nu[F]$ for a probability measure $\nu$ satisfying $\gcb{D}$, we have \eqref{eq:expcon1} and \eqref{eq:expcon2} with $K=2D\|\ushort{\delta} (F) \|_2^2$, thus we get \eqref{eq:momentb} with $C_{2p}=p! (8D)^p$.
We refer to \cite[Theorem 2.1, p. 25]{blm} for a proof of these two general statements.

\section{\textbf{Gibbs measures}}\label{sec:Gibbs}

For the sake of convenience, we briefly recall some facts about Gibbs measures which will be used later on.
We refer to \cite{Geo} for details. The largest class of potentials we consider is that of shift-invariant
``uniformly summable'' potentials.

\subsection{Potentials}

A potential is a function $\Phi:\mathcal{P}\times \Omega\to\R$. (Recall that $\mathcal{P}$ is  the collection of finite subsets of $\Zd$.)
We will assume that $\omega\mapsto \Phi(\Lambda,\omega)$ is $\mathfrak{B}_\Lambda$-measurable for every
$\Lambda\in\mathcal{P}$. Shift-invariance is the requirement that
$\Phi(\Lambda+x,T_x\omega)=\Phi(\Lambda,\omega)$ for all $\Lambda\in\mathcal{P}$, $\omega\in\Omega$
and $x\in\Zd$ (where $\Lambda+x=\{y+x: y\in\Lambda\}$). 
Uniform summability is the property that 
\begin{equation}\label{def-sumpot}
\opnorm{\Phi}:=\sum_{\substack{\Lambda\in\mathcal{P}\\ \Lambda \ni 0}}
\|\Phi(\Lambda,\cdot)\|_\infty <\infty.
\end{equation}
We shall denote by $\mathscr{B}_T$ the space of uniformly summable shift-invariant continuous potentials. 
Equipped with the norm $\opnorm{\cdot }$, it is a Banach space.

The most important subclass of uniformly summable shift-invariant potentials is the class of finite-range potentials.
A finite-range potential is such that there exists $R>0$ such that $\Phi(\Lambda,\omega)=0$ if $\textup{diam}(\Lambda)>R$.
The smallest such $R$ is called the range of the potential.
More formally, $R=R(\Phi)=\max_{\Lambda:\Phi(\Lambda,\cdot)\not\equiv 0}\textup{diam}(\Lambda)$.
Nearest-neighbor potentials correspond to the case $R=1$.
The set of potentials with finite range is dense in $\mathscr{B}_T$.

Now define the continuous function
\begin{equation}\label{def-fU}
f_\Phi(\omega)=\sum_{\Lambda\ni 0} \frac{\Phi(\Lambda,\omega)}{|\Lambda|}.
\end{equation}
The quantity $f_\Phi(\omega)$ can be interpreted as the mean energy per site in the configuration $\omega$.

\subsection{Gibbs measures}\label{sec:gibbs}

Given $\Phi\in \mathscr{B}_T$ and $\Lambda\in\mathcal{P}$, the associated Hamiltonian in the finite volume $\Lambda$ with boundary condition $\eta\in\Omega$ is given by
\[
\mathcal{H}_{\Lambda}(\omega|\eta)= \sum_{\substack{\Lambda'\in \mathcal{P}\\ \Lambda'\cap \Lambda \neq\emptyset}} \Phi(\Lambda',\omega_{\Lambda}\eta_{\Zd\backslash \Lambda})\,.
\]
The corresponding specification is then defined as
\begin{equation}\label{eq:spe}
\bgamma^{\Phi}_{\Lambda} (\omega|\eta)
=\frac{\exp\left(-\mathcal{H}_{\Lambda}(\omega|\eta)\right)}{Z_{\Lambda}(\eta)}\,
\end{equation}
where $Z_{\Lambda}(\eta)$ is the partition function in $\Lambda$ (normalizing factor).
We say that $\gmu$ is a Gibbs measure for the potential $\Phi$ if 
$\bgamma^{\Phi}_{\Lambda} (\omega|\cdot)$ is a version of the conditional probability
$\gmu(\omega_{\Lambda}| \mathfrak{B}\!_{\Lambda^c})$. 
Equivalently, this means that for all $A\in \mathfrak{B}$, $\Lambda \in \mathcal{P}$, one has the so-called
``DLR equations''
\begin{equation}\label{eq:dlr}
\mu(A)=\int \dd\mu(\eta)  \sum_{\omega'\in\Lambda} \bgamma^{\Phi}_{\Lambda} (\omega'|\eta) \, \un\!_{A}(\omega'_{\Lambda}\eta_{\Lambda^c}).
\end{equation}
A consequence of \eqref{def-sumpot} is that for all $\Lambda\supset \Lambda'$ such
that $\Lambda\in\mathcal{P}$, for all $\omega,\tilde\omega$ such that $\omega_x=\tilde\omega_x$ $\forall
x\notin \Lambda'$, we have
\[
\sup_{\eta\in\Omega} \big|\mathcal{H}_{\Lambda}(\omega|\eta)-\mathcal{H}_{\Lambda}(\tilde\omega|\eta)\big|
\leq 2 \, \sum_{A\cap \Lambda' \neq \emptyset} \| \Phi(A,\cdot)\|_\infty \leq 2\, |\Lambda'| \, \opnorm{\Phi}.
\]
As a further consequence we get
\begin{equation}\label{eq:bruit}
\frac{\bgamma^{\Phi}_{\Lambda} (\omega|\eta)}{\bgamma^{\Phi}_{\Lambda} (\tilde\omega|\eta)} \leq
\exp(2 |\Lambda'| \, \opnorm{\Phi}).
\end{equation}

The set of Gibbs measures for a given potential is never empty but it may be not reduced to a singleton. This set necessarily contains at least one Gibbs measure that is shift invariant.

Finally, let
\begin{equation}\label{PPhi}
P(\Phi)=\lim_{n\to\infty} \frac{1}{(2n+1)^d} \log Z_{C_n}(\eta^{(n)})
\end{equation}
which exists for any sequence $(\eta^{(n)})_{n\geq 1}$ and depends only on $\Phi$.
At certain places in the sequel, we will need a good control on the measure of cylinders in terms of the ergodic sum
of $f_\Phi$. To ensure this we will have to assume additionally that $\Phi$ satisfies 
\begin{equation}\label{cond-decay}
\sum_{n\geq 1} n^{d-1} \sum_{\substack{\Lambda \ni 0\\\Lambda \cap (\Zd\backslash C_n)\neq \emptyset}}
\frac{\|\Phi(\Lambda,\cdot)\|_\infty}{|\Lambda|}<\infty.
\end{equation}
This condition is obviously satisfied by any finite-range potential, but also by a class of spin pair potentials
(see below). This condition implies 
\[
\sum_{n\geq 1} n^{d-1}\, \text{var}_n(f_\Phi)<\infty
\]
where $\text{var}_n(f_\Phi):=\sup\{|f_\Phi(\omega)-f_\Phi(\omega')| : \omega_{C_n}=\omega'_{C_n}\}$. 
From \cite[Theorem 5.2.4, p. 100]{keller} it follows that there exists $C_\Phi>0$ such that for all $\omega\in\Omega$ and for all $n\in\N$, one has
\begin{equation}\label{eq:control-cyl}
e^{-C_\Phi (2n+1)^{d-1}}\leq
\frac{\gmu_\Phi(\EuScript{C}_n(\omega))}{e^{-(2n+1)^d P(\Phi)-\sum_{x\in C_n}f_\Phi(T_x \omega)}}
\leq e^{C_\Phi (2n+1)^{d-1}}.
\end{equation}
The point, which we will need later, is that, under \eqref{cond-decay}, we have {\em surface-order} terms in the exponentials on both sides.

\subsection{Entropy, relative entropy and the variational principle}

The entropy (per site) of a shift-invariant probability measure $\nu$ is defined as
\[
h(\nu)=\lim_{n\to\infty} -\frac{1}{(2n+1)^d} \sum_{\omega\in\Omega_n} \nu_n(\omega)\log\nu_n(\omega)
\;(\in [0,\log|S|])
\]
where $\nu_n$ is the probability measure induced on $\Omega_n$ by projection, {\em i.e.}, 
$\nu_n(\omega)=\nu(\EuScript{C}_n(\omega))$.

Given two probability measures $\mu$ and $\nu$ on $\Omega$, let
\[
H_n(\nu_n|\mu_n)=
\sum_{\omega\in \Omega_n} \nu_n(\omega) \log \frac{\nu_n(\omega)}{\mu_n(\omega)}.
\]
It can be proven \cite[Chapter 15]{Geo} that if $\nu$ is a shift-invariant probability measure and $\mu$ a 
Gibbs measure, we can define the relative entropy density of $\nu$ with respect to $\mu$ as
\begin{equation}\label{def:relent} 
h(\nu|\mu)=\lim_{n\to\infty} \frac{H_n(\nu_n|\mu_n)}{(2n+1)^d}.
\end{equation}
One has $h(\nu|\mu)\in [0,+\infty)$. Moreover, if $\Phi\in\mathscr{B}_T$ and $\mu_\Phi$ is a shift-invariant
Gibbs measure for $\Phi$ then
\begin{equation}\label{formule-rel-entropy}
h(\nu|\mu_\Phi)=P(\Phi)+\E_\nu[f_\Phi]-h(\nu).
\end{equation}

Finally, the variational principle (\cite[Chapter 15]{Geo}) states that
$h(\nu|\mu_\Phi)=0$ if and only if $\nu$ is a Gibbs measure for $\Phi$.
In particular, for such a $\nu$, one has
\begin{equation}\label{VP}
P(\Phi)=h(\nu)-\E_{\nu}[f_\Phi].
\end{equation}

\subsection{Examples}\label{sec:examples}

In order to make things more tangible, we will repeatedly illustrate our results with the following concrete examples.

\begin{enumerate}
\item[]\textbf{(Ising)}
A fundamental example is the (nearest-neighbor) Ising model for which we take $S=\{-1,+1\}$ and
that we define via the nearest-neighbor potential
\begin{equation}\label{def-ising}
\Phi(\Lambda,\omega)=
\begin{cases}
- h\omega_x & \textup{if}\quad \Lambda=\{x\}\\
- J \omega_x\omega_y & \textup{if}\quad \Lambda=\{x,y\}\;\textup{and}\;\|x-y\|_1=1\\
0 & \textup{otherwise}
\end{cases}
\end{equation}
where the parameters $J,h\in \R$  are respectively the coupling strength and the external magnetic field (uniform with strength $|h|$). When $J>0$, this is called the ferromagnetic case, when $J<0$ it is called the antiferromagnetic case.
We shall consider the potential $\beta\Phi$, where $\beta\in\R_+$ is the inverse temperature.
\item[]\textbf{(Long-range Ising)}
Sticking to the case $S=\{-1,+1\}$, one can define the so-called spin pair potentials that can be of infinite range.
Let $J:\Zd\to \R$ be an even function such that $J(0)=0$ and $0<\sum_{x\in\Zd} |J(x)|<+\infty$. Then define
\begin{equation}\label{def-sum-pair-pot}
\Phi(\Lambda,\omega)=
\begin{cases}
- J(x-y)\, \omega_x\omega_y & \textup{if}\quad \Lambda=\{x,y\}\\
0 & \textup{otherwise.}
\end{cases}
\end{equation}
When $J$ is positive-valued, we have a ferromagnetic spin pair potential, while when $J$ is negative-valued, we have an anti-ferromagnetic spin pair potential. 
For this class of potentials, the following facts are known \cite{ellis} in the ferromagnetic case. Let $\mathscr{J}_0:=\sum_{x\in\Zd} J(x)$ (which is finite by assumption). Then $\mathscr{J}_0^{-1}\leq \beta_c:=\sup\{\beta>0 : \E_{\mu_{\beta\Phi}}[s_0]=0\}$, where $s_0(\omega)=\omega_0$.
Moreover, if there exist two linearly independent unit vectors $z,z'$ in $\Z^d$ such that
$J(z)$ and $J(z')$ are positive, then $\beta_c$ is finite. Of course, this class contains the nearest-neighbor Ising model with zero external magnetic field.
\item[]\textbf{(The Potts antiferromagnet)}
Another example of a nearest-neighbor potential is the Potts antiferromagnet for which $S=\{1,2,\ldots,q\}$ where $q$ is an integer  greater than or equal to $2$. The elements of $S$ are traditionally viewed as `colors'. The potential is defined as
\begin{equation}
\label{def-potts}
\Phi(\Lambda,\omega)=
\begin{cases}
J \un_{\{\omega_x=\omega_y\}} & \textup{if}\quad \Lambda=\{x,y\}\;\textup{and}\;\|x-y\|_1=1\\
0 & \textup{otherwise}
\end{cases}
\end{equation}
where $J>0$ is the coupling strength. (For $q=2$, this potential is physically equivalent to the Ising potential.) One can add an external magnetic field as in the Ising model.
\end{enumerate}

\section{\textbf{Gaussian concentration bound for Gibbs measures}}\label{sec:gcbgauss}

The Gaussian concentration property holds under the Dobrushin uniqueness condition.
In view of the applications to come, we give concrete
examples of potentials satisfying this condition.

\subsection{Dobrushin uniqueness regime}\label{sec:dob}

Let $\Phi\in\mathscr{B}_T$ and $\bgamma^{\Phi}$ be the corresponding specification.
The Dobrushin uniqueness condition is based upon the matrix
\[
C_{x,y}(\bgamma^{\Phi})=
\sup_{\omega,\omega'\in\Omega: \, \omega_{\Zd\backslash y}=\, \omega'_{\Zd\backslash y}}
\| \bgamma^{\Phi}_{\{x\}}(\cdot |\omega)-\bgamma^{\Phi}_{\{x\}}(\cdot|\omega')\|_\infty.
\]
Because we consider shift-invariant potentials, $C_{x,y}(\bgamma^{\Phi})$ depends only on $x-y$.
One says that $\bgamma^{\Phi}$ satisfies the Dobrushin uniqueness condition if
\begin{equation}\label{def:dobcoef}
\mathfrak{c}(\bgamma^{\Phi}):=\sum_{x\in\Zd} C_{0,x}(\bgamma^{\Phi})<1.
\end{equation}
It is well known (see {\em e.g.} \cite[chapter 8]{Geo}) that if this condition holds, there is a unique Gibbs measure for
$\Phi$ which we denote by $\mu_\Phi$. Moreover it is automatically shift invariant.

\subsection{Examples}

The following list of examples is not exhaustive. All details can be found in \cite[Chapter 8]{Geo}.

Let $\Phi\in\mathscr{B}_T$. One has the bound
\[
\mathfrak{c}(\bgamma^{\Phi}) \leq \frac{1}{2}\sum_{\Lambda\ni 0} \, (|\Lambda|-1)\, \delta(\Phi(\Lambda,\cdot))
\]
where 
\[
\delta(\Phi(\Lambda,\cdot))=\sup_{\omega,\omega'\in\Omega} |\Phi(\Lambda,\omega)-\Phi(\Lambda,\omega')|.
\]
Hence a sufficient condition for \eqref{def:dobcoef} to hold is that
\begin{equation}\label{def-high-temp}
\sum_{\Lambda\ni 0} \, (|\Lambda|-1)\, \delta(\Phi(\Lambda,\cdot))<2.
\end{equation}

Let us come back to the examples introduced above.
As a first example, take a potential $\beta\Phi$ where $\beta>0$ and $\Phi$ is a finite-range potential.
It is obvious that \eqref{def-high-temp} holds for all $\beta$ small enough.
In this case it is customary to say that we are in the ``high-temperature regime'' of this potential.
A second scenario is when we have a sufficiently large external magnetic field. By this we mean 
that we take any potential $\Phi$ such that $\Phi(\{x\},\omega)=-h\, \omega_x$ for all $x\in\Zd$ and some
$h\in\R$. The condition implying \eqref{def:dobcoef} reads
\[
e^{|h|}> \exp\left( \frac{1}{2} \sum_{\Lambda\ni 0 : |\Lambda|>1} \delta(\Phi(\Lambda,\cdot))\right)\, \sum_{\Lambda\ni 0} \, (|\Lambda|-1)\, \delta(\Phi(\Lambda,\cdot)).
\]
A third scenario occurs at low temperatures for potentials with unique ground state, {\em e.g.}, the Ising model with $h\neq 0$ and for sufficiently large $\beta$, or any $\beta$ and $|h|$ sufficiently large.
\begin{enumerate}
\item[]\textbf{(Ising)}
For instance, in the Ising model in two dimensions, \eqref{def:dobcoef} holds if $|h|>4\beta|J|+\log(8\beta|J|)$.
Without external magnetic field ($h=0$) and with $J=1$, \eqref{def:dobcoef} holds if $\beta<\frac{1}{2}\ln(\frac{5}{3})\approx 0.255$.
\item[]\textbf{(Long-range Ising)}
For a spin pair potential $\beta\Phi$ one has
\[
\mathfrak{c}(\bgamma^{\beta\Phi}) \leq \sum_{x\in\Zd}\tanh(\beta |J(x)|),
\]
hence \eqref{def:dobcoef} holds if
\begin{equation}\label{eq:tanh}
\sum_{x\in\Zd}\tanh(\beta |J(x)|)< 1.
\end{equation}
This holds in particular if $\sum_{x\in\Zd} \beta |J(x)|\leq  1$.
\item[]\textbf{(Potts antiferromagnet)}
Potts antiferromagnet \eqref{def-potts} satisfies Dobrushin's uniqueness condition as soon as $q>4d$,
regardless of the value of $J$. Indeed, one can check that $\mathfrak{c}(\bgamma^{\Phi}) \leq \frac{2d}{q-2d}$. We refer to \cite{salassokal} for this result which improves the one described in \cite{Geo}.
Moreover, in that regime, for the unique Gibbs $\mu_\Phi$ it holds that $\E_{\mu_\Phi}[\un_{\{\omega_0=i\}}]=1/q$ for $i\in \{1,\ldots,q\}$.
\end{enumerate}

\subsection{Gaussian concentration bound}

\begin{theorem}[\cite{CCKR,Kul}]\label{thm:gemb-dob}
\leavevmode\\ 
Let $\Phi\in\mathscr{B}_T$ and assume that the associated specification $\bgamma^{\Phi}$ satisfies  
Dobrushin's uniqueness condition \eqref{def:dobcoef}.
Then $\mu_\Phi$ satifies $\gcb{\frac{1}{2(1-\mathfrak{c}(\bgamma^{\Phi}))^2}}$.
\end{theorem}

Take for instance a spin pair potential satisfying \eqref{eq:tanh}. Then, \eqref{eq:expcon2}
gives
\begin{align*}
\MoveEqLeft[6] \gmu_{\beta\Phi}\left\{\omega\in \Omega : |F(\omega) - \E_{\mu_{\beta\Phi}}[F]| \geq u \right\} \\
& \leq
2\exp\left(-\frac{\big(1-\sum_{x\in\Zd}\tanh(\beta |J(x)|)\big)^2u^2}{2\|\ushort{\delta} (F) \|_2^2}\right)
\end{align*}
for all functions $F\in\Delta_2(\Omega)$ and for all $u>0$. Observe that when $\beta$ goes to $0$, 
$\gmu_{\beta\Phi}$ goes (in weak topology) to a product measure (namely the product of the 
measures giving equal mass to each element of $S$), and one gets
$-\frac{u^2}{2\|\ushort{\delta} (F) \|_2^2}$ in the exponential. 

\begin{remark}\label{kulkul}
Theorem \ref{thm:gemb-dob} was first proved in \cite{Kul} in a more general setting (in particular, without assuming that potentials are shift invariant). 
Using a different approach, this theorem was also proved in \cite[Section 3.1]{CCKR} for shift-invariant potentials, although it was not explicitly stated therein. In particular, the constant is not explicit. Moreover, it was proved for local functions.
But it is not difficult to show that, if $\gcb{D}$ holds for all local functions, then it holds for all functions in $\Delta_2(\Omega)$ with the same constant $D$, as shown in the lemma below.
\end{remark}

\begin{lemma}\label{extension}
If \eqref{eq:gemb} holds with constant $D$, then it holds for 
all $F\in\Delta_2(\Omega)$ with the same constant $D$.
If \eqref{eq:momentb} holds for some $p\geq 1$ with a constant
$C_{2p}$, then it extends to this class of functions, with the same constant.
\end{lemma}

\begin{proof}
We treat the case of the Gaussian concentration bound. The case of moment bounds is very similar.
Let $F:\Omega\to\R$ be a continuous function such that $\|\ushort{\gdelta}(F)\|_2<+\infty$.
Since $\Omega$ is compact, $F$ is bounded, thus
$\E_\nu[\exp(F)]<+\infty$. We now construct a sequence of local functions $(F_n)_n$
defined in the natural way: We fix once for all $\eta\in\Omega$ and for each $n\geq 1$ we let
\[
F_n(\omega)=F(\omega_{C_n}\eta_{\Zd\backslash C_n}),
\]
that obviously coincides with $F$ inside the cube $C_n$. We now prove that 
$\|\ushort{\gdelta}(F-F_n)\|_2\to 0$ as $n\to+\infty$. We first prove that, for each $x\in\Zd$, $\delta_x(F_n-F)\xrightarrow[]{n\to\infty}0$.
Since $x$ is fixed and $n$ gets arbitrarily large, we can assume that $x\in C_n$. We have
\[
\delta_x(F_n-F)=\sup\{|F_n(\omega)-F(\omega')| : \omega_y=\omega'_y,\, \forall y\neq x\}.
\]
By compactness, there exists two configurations
$\omega=\omega_{C_n\backslash\{x\}}s_x\omega_{C_n^c}$ and 
$\omega'=\omega_{C_n\backslash\{x\}}s'_x\omega_{C_n^c}$ such that this supremum
is attained. (The notation should be clear: given
$\omega\in\Omega$,  $\omega_{C_n\backslash\{x\}}s_x\omega_{C_n^c}$ is the configuration 
coinciding with $\omega$ except at site $x\in C_n$ where $\omega_x$ is replaced by $s\in S$ at
site $x$.)
Therefore
\begin{align*}
& \delta_x(F_n-F)\leq 
|F(\omega_{C_n\backslash\{x\}}s_x\eta_{C_n^c})-F(\omega_{C_n\backslash\{x\}}s_x\omega_{C_n^c})| \\
& \qquad\qquad\qquad + |F(\omega_{C_n\backslash\{x\}}s'_x\omega_{C_n^c})-F(\omega_{C_n\backslash\{x\}}s'_x\eta_{C_n^c})|.
\end{align*}
By continuity, the two terms go to zero as $n$ goes to infinity. Then we obviously have that $(\delta_x(F_n-F))^2\leq 4 (\delta_x(F))^2$. Since $\sum_{x\in \Zd} (\delta_x(F))^2<\infty$, we can apply the dominated convergence theorem for sums to get the desired conclusion.

Now  \eqref{eq:gemb} follows for $F$ with the same constant, because $\|F-F_n\|_\infty\to 0$ and 
\begin{align*}
\MoveEqLeft  \E_\nu\big[\exp\left(F - \E_\nu[F]\right) \big] \\
& \leq \E_\nu\big[\exp\left(F_n - \E_\nu[F_n]\right) \big] \, \exp\left(2\|F-F_n\|_\infty\right)  \\
& \leq \exp\left(D \|\ushort{\delta}(F_n) \|_2^2\right)\, \exp\left(2\|F-F_n\|_\infty\right) \\
& \leq \exp\left(D \|\ushort{\delta}(F) \|_2^2\right)  \\
&  \;\quad \times \exp\left(2\|F-F_n\|_\infty + 2\|\ushort{\delta}(F)\|_2\|\ushort{\delta}(F_n-F)\|_2 +
\|\ushort{\delta}(F_n-F)\|_2^2\right).
\end{align*}
This result now follows by taking the limit $n\to \infty$ in the right-hand side.
\end{proof}

\section{\textbf{Concentration bounds for the Ising ferromagnet at low temperature}}\label{sec:lowtempconc}

\subsection{The Ising ferromagnet}

We consider the low-temperature plus-phase of the Ising model on $\Z^d$, $d\geq 2$, corresponding
to the potential \eqref{def-ising} with $h=0$, $J>0$ (ferromagnetic case) and the boundary condition
$\eta_x=+1$ for all $x\in\Zd$.
Without loss of generality, we can take $J=1$.
This is the probability measure $\gmu^+_\beta$ on $\Omega$ defined as the weak limit as
$\Lambda\uparrow\Z^d$ of the finite-volume measures 
\begin{equation}\label{rat}
\gmu^+_{\Lambda,\beta}(\omega_\Lambda) =\frac{1}{Z_{\Lambda,\beta}(+^{\Zd})}\
\exp\Big(-\beta \mathcal{H}_{\Lambda}\big(\omega\, |+^{\Zd}\big)\Big)
\end{equation}
where 
\[
\mathcal{H}_{\Lambda}\big(\omega|+^{\Zd}\big)=
-\sum_{\substack{x,y\in \Lambda\\ \|x-y\|_1=1}} \omega_x\omega_y - \sum_{\substack{x\in\partial
\Lambda, \ y\notin \Lambda\\ \|x-y\|_1=1}}\omega_x
\]
and where $\beta\in\R^+$. We write $+^{\Zd}$ for the configuration $\eta$ such that
$\eta_x=+1$ for all $x\in\Zd$, and $\partial \Lambda$ denotes the inner boundary of the set $\Lambda$, {\em i.e.}, the set of those $x\in \Lambda$ having at least one neighbor $y\notin \Lambda$.
The existence of the limit $\Lambda\uparrow \Z^d$ of $\gmu_{\Lambda,\beta}^+$ is by a standard and well-known monotonicity argument, see {\em e.g.} \cite{Geo}. 
In a similar fashion one can define $\gmu_\beta^-$. 
Both $\gmu_\beta^+$ and $\gmu_\beta^-$ are shift-invariant and ergodic.
It is well known that there exists $\beta_c>0$ such that for all $\beta>\beta_c$, $\gmu_\beta^+\neq\gmu_\beta^-$.

\subsection{Moment concentration bounds of all orders}

It should not be a surprise that, for the Ising model in the phase coexistence region, a Gaussian concentration bound cannot hold. Indeed, this would contradict the surface-order large deviations for the magnetization in that regime (see below for more details). Nevertheless, one can control all moments, as was shown in \cite{CCKR}.

\begin{theorem}[\cite{CCKR}]\label{thm:allmomentsIsing}
\leavevmode\\
Let $\gmu_\beta^+$ be the plus phase of the low-temperature Ising model defined above. There exists
$\bar{\beta}>\beta_c$, such that for each $\beta>\bar{\beta}$, there exists a positive sequence
$(C_{2p}(\beta))_{p\in\N}$ such that the measure $\gmu_\beta^+$ satifies
$\mcb{2p,C_{2p}(\beta)}$ for all $p\in\N$. In particular one has for each $p\in\N$
\[
\gmu_\beta^+\left\{\omega\in \Omega :  |F(\omega)-\E_{\mu_\beta^+} [F]|\geq u\right\} \leq 
\frac{C_{2p}(\beta)\, \| \ushort{\gdelta}(F)\|_2^{2p}}{u^{2p}}
\]
for all functions $F\in\Delta_2(\Omega)$ and for all $u>0$.
\end{theorem}

\begin{remark}\label{rem-moments-lowT}
In view of Subsection \ref{rem:moments-vs-gb}, one can ask whether the previous theorem implies in fact a stronger statement, namely a Gaussian concentration bound. The answer turns out to be negative. Indeed, looking at the proof of Theorem 3 in \cite{CCKR}, one sees that $C_{2p}$ is of the form $p^{2p} K^p$ for some constant $K>0$ (depending on $F$ but independent of $p$). Therefore, one cannot infer a Gaussian bound from these moment bounds. 
\end{remark}

\subsection{Stretched-exponential concentration bound}

One can deduce from the previous theorem that the measure $\gmu_\beta^+$
satisfies a `stretched-exponen\-tial' concentration bound. This was shown in \cite{CCKR}.
In order to state it, we need some notations and definitions.
For $0<\varrho<1$, let $M\!_\varrho:\R\to\R^+$ be the Young function defined by
$M\!_\varrho(x)=e^{(|x|+h_\varrho)^\varrho}-e^{h_\varrho^\varrho}$ where $h_\varrho=(\frac{1-\varrho}{\varrho})^{1/\varrho}$. Then, the Luxemburg norm with respect to $M_\varrho$ of a real-valued random variable $Z$ is defined by
\[
\|Z\|_{M\!_\varrho}=\inf\left\{ \lambda>0 : \E\left[M\!_\varrho\left(\frac{Z}{\lambda}\right)\right]\leq 1\right\}.
\]
(Note that the choice $M_p(x)=|x|^p$ would give the usual $L^p$ norm.)

\begin{theorem}[\cite{CCKR}]\label{thm-seb}
\leavevmode\\
Let $\gmu_\beta^+$ be the plus-phase of the low-temperature Ising model and $\bar{\beta}$ as in the previous theorem.
Then, for each $\beta>\bar{\beta}$, there exist $\varrho=\varrho(\beta)\in(0,1)$ and a constant $K_{\!\varrho}>0$ such that, for all functions $F\in\Delta_2(\Omega)$, one has
\begin{equation}\label{eq:lux-ising}
\|F-\E_{\mu_\beta^+}[F]\|_{M_\varrho}\leq K_{\!\varrho}\, \|\ushort{\gdelta}(F)\|_2,
\end{equation}
Moreover there exists $c_\varrho>0$ such that for all $u>0$
\begin{equation}\label{eq:devineqsubexp}
\gmu_\beta^+\left\{\omega\in \Omega :  |F(\omega)-\E_{\mu_\beta^+} [F]|\geq u\right\} \leq
4\, \exp\left(-\frac{c_\varrho\, u^{\varrho}}{\|\ushort{\delta}(F)\|_2^\varrho}\right).
\end{equation}
\end{theorem}

All the constants appearing in the previous statement may depend on $d$.

Theorems \ref{thm:allmomentsIsing} and \ref{thm-seb} were proved in \cite{CCKR} for local functions,
but Lemma \ref{extension} shows that their extension to functions in
$\Delta_2(\Omega)$ is ensured.

\begin{remark}
For any random variable $Z$ and for any $0<\varrho <1$, there is a real number $B_\varrho>1$ such that, if $\|Z\|_{M_\varrho}<\infty$, then
\begin{equation}\label{LpOrlicz}
B_\varrho^{-1} \sup_{q\in 2\N} \frac{\|Z\|_{L^q}}{q^{1/\varrho}}
\leq 
\|Z\|_{M_\varrho}
\leq
B_\varrho \sup_{q\in 2\N} \frac{\|Z\|_{L^q}}{q^{1/\varrho}}.
\end{equation}
These estimates are proved in \cite[p. 86]{machkouri} where the suprema are taken over all
the integers greater than $2$. Restricting the supremum to even integers gives the same inequalities with slightly different constants.
\end{remark}

\begin{remark}
An essential ingredient in the proofs of Theorems \ref{thm:allmomentsIsing} and \ref{thm-seb} is 
a non-trvial coupling constructed in \cite{MRSvM}. In fact, this construction was made for Markov random fields for which the Pirogov-Sinai theory applies, such as the low-temperature pure phases of the ferro- and anti-ferromagnetic Potts model. For the sake of simplicity, only the ferromagnetic Ising model was considered in
\cite{CCKR}. Therefore we also restrict ourselves to this case in the present work.
\end{remark}

\section{Application 1: Ergodic sums and empirical pair correlations}\label{sec:app1}

\subsection{General results}

Given a nonempty finite subset $\Lambda$ of $\Zd$ ({\em i.e.}, $\emptyset\neq\Lambda\in\mathcal{P}$), a continuous function $f:\Omega\to\R$ and $\omega\in\Omega$, define
\[
S_\Lambda f(\omega)=\sum_{x\in\Lambda} f(T_x \omega).
\]
A sequence $(\Lambda_n)_n$ of nonempty finite subsets of $\Zd$ is said to tend to infinity in the sense of van Hove if, for each $x\in \Zd$, one has
\[
\lim_{n\to+\infty} |\Lambda_n|=+\infty\quad\text{and}\quad\lim_{n\to+\infty} \frac{|(\Lambda_n+x)\backslash \Lambda_n|}{|\Lambda_n|}=0.
\]
In the language of countable discrete amenable groups, $(\Lambda_n)_n$ is a F\o lner sequence.
A special case of interest is when $\Lambda_n=C_n$:
\[
S_n f(\omega):=\sum_{x\in C_n} f(T_x\omega), \, n\in\N.
\]
By convention we set $S_0 f(\omega)=f(\omega)$.
Given an ergodic measure $\nu$, we are interested in the fluctuations of
\[
\frac{S_\Lambda f(\omega)}{|\Lambda|}.
\]
When one considers 
\[
\frac{S_{\Lambda_n} f(\omega)}{|\Lambda_n|}
\]
where $(\Lambda_n)_n$ tends to infinity in the sense of van Hove, it is well-known that this average
converges $\nu$-almost surely to $\E_\nu[f]$ as $n\to+\infty$. This is the so-called multidimensional ergodic theorem, see {\em e.g.} \cite{tempelman}. 

We first state a simple lemma that will be repeatedly used in this
section and later.

\begin{lemma}\label{lem:deltasums}
Let $f\in \Delta_1(\Omega)$ and $\Lambda\in\mathcal{P}$. Then 
\[
\| \ushort{\gdelta}(S_\Lambda f)\|_2^2\leq |\Lambda | \, \| \ushort{\gdelta}(f)\|_1^2.
\]
\end{lemma}
\begin{proof}
We observe that $\delta_z(S_\Lambda f)\leq \sum_{x\in \Lambda}\delta_{z-x}(f)$.
We now use Young's inequality:
if $\ushort{\boldsymbol{u}}=(\boldsymbol{u}_x)_{x\in\Zd}\in \ell^p(\Zd)$ and
$\ushort{\boldsymbol{v}}=(\boldsymbol{v}_x)_{x\in\Zd}\in\ell^q(\Zd)$, where $p,q\geq 1$, then $\ushort{\boldsymbol{u}}* \ushort{\boldsymbol{v}}\in\ell^r(\Zd)$
where $r\geq 1$ is such that $1+r^{-1}=p^{-1}+q^{-1}$, and
\[
\|\ushort{\boldsymbol{u}}* \ushort{\boldsymbol{v}}\|_{\ell^r(\Zd)}\leq 
\|\ushort{\boldsymbol{u}}\|_{\ell^p(\Zd)}\|\ushort{\boldsymbol{v}}\|_{\ell^q(\Zd)}.
\]
We apply this inequality with $r=2, p=2,q=1$, $\boldsymbol{u}_x=\un_{\Lambda}(x)$,
and $\boldsymbol{v}_x=\delta_x(f)$ to get the desired estimate.
\end{proof}

We get immediately the following general result.

\begin{theorem}\label{thm:ergo}
\leavevmode\\
Let $\nu$ be a shift-invariant probability measure satisfying $\gcb{D}$.
Then for all $\Lambda\in\mathcal{P}$ and for all $f\in\Delta_1(\Omega)$ we have, for all $u>0$,
\begin{equation}\label{eq:conc-ergo}
\nu\left\{ \omega\in\Omega : \left| \frac{S_\Lambda f(\omega)}{|\Lambda|} - \E_\nu[f]\right|\geq u\right\}
\leq 2\, \exp\left(-\frac{|\Lambda|\, u^2}{4D\, \| \ushort{\delta}(f)\|_1^2}\right)\,.
\end{equation}
\end{theorem}

\medskip

Two functions are of particular interest in the context of Gibbs measures:
\begin{itemize}
\item[(a)]
Magnetization: For $S=\{-1,+1\}$ let $f=s_0$ where $s_0(\omega)=\omega_0$. Then,
for a given $\Lambda\in\mathcal{P}$, define
\[
M_\Lambda(\omega)=\sum_{x\in \Lambda}s_0(T_x \omega)\,,
\]
which is the empirical (total) magnetization in $\Lambda$. We have
\[
\|\ushort{\delta}(s_0)\|_1=2.
\]
\item[(b)]
Mean energy per site:  Take $f=f_\Phi$ where $\Phi\in\mathscr{B}_T$. From \eqref{def-fU} we get
\[
\delta_x(f_\Phi)\leq 2\, \sum_{\substack{\Lambda\ni 0\\ \Lambda\ni x }} \frac{\|\Phi(\Lambda,\cdot)\|_\infty}{|\Lambda|}
\]  
As a consequence we have
\[
\|\ushort{\delta}(f_\Phi)\|_1\leq 2\, \opnorm{\Phi}.
\]
\end{itemize}

\subsection{Empirical magnetization and energy in Dobrushin's uniqueness regime}

Applying Theorem \ref{thm:ergo} to the previous two functions gives the following two results.

\begin{theorem}\label{thm:two-high-results}
\leavevmode\\
Let $\bgamma^{\Phi}$ be the specification of a potential $\Phi\in\mathscr{B}_T$ satisfying 
Dobrushin's uniqueness condition \eqref{def:dobcoef}. Then, for all $\Lambda\in\mathcal{P}$, 
we have
\begin{itemize}
\item[\textup{(a)}] the concentration bound
\[
\gmu_\Phi\left\{\omega\in \Omega : \left|\frac{M_\Lambda(\omega)}{|\Lambda|} -  \E_{\mu_\Phi}[s_0]\right| \geq u  \right\}
\leq 2 \exp\left(-c\, |\Lambda|\, u^2\right)
\] 
for all $u>0$, where
\[
c=\frac{ (1-\mathfrak{c}(\bgamma^{\Phi}))^2}{8},
\]
\item[\textup{(b)}]
and for all $\Psi\in \mathscr{B}_T$, the concentration bound
\[
\gmu_\Phi\left\{\omega\in \Omega :
\left|\frac{S_\Lambda f_\Psi(\omega)}{|\Lambda |} - \E_{\mu_\Phi}[f_\Psi]\right| \geq u  \right\}\\
\leq 2 \exp\left(-c\, |\Lambda|\, u^2\right)
\]
for all $u>0$, where
\[
c=\frac{ (1-\mathfrak{c}(\bgamma^{\Phi}))^2}{8\, \opnorm{\Psi}^2}.
\]
\end{itemize}
\end{theorem}

We refer back to Section \ref{sec:examples} (which contains our three main examples) if the reader wants to make the previous bounds even more explicit.

\subsection{Empirical magnetization and energy in the low-temperature Ising model}

For the plus-phase of the low-temperature Ising model we can apply Theorem
\ref{thm-seb} to obtain the following analogue of Theorem \ref{thm:two-high-results}.

\begin{theorem}[] \label{thm:two-low-results}
\leavevmode\\
Let $\gmu_\beta^+$ be the plus phase of the low-temperature Ising model. 
Then there exists  $\bar{\beta}>\beta_c$ such that, for each $\beta>\bar{\beta}$, there exist $\varrho=\varrho(\beta)\in(0,1)$ and a constant $c_\varrho>0$ such that, for all $\Lambda\in\mathcal{P}$,  we have
\begin{itemize}
\item[\textup{(a)}] the concentration bound
\[
\gmu_\beta^+\left\{\omega\in \Omega :  \left|\frac{M_\Lambda(\omega)}{|\Lambda|}  - 
\E_{\mu_\beta^+}[s_0]\right|\geq u\right\} \leq
4\, \exp\left(-\frac{c_\varrho}{2^\varrho} |\Lambda|^{\frac{\varrho}{2}} u^{\varrho}\right),
\]
for all $u>0$,
\item[\textup{(b)}] and, for all $\Psi\in \mathscr{B}_T$, the concentration bound
\[
\gmu_\beta^+\left\{\omega\in \Omega :  \left|\frac{S_\Lambda f_\Psi(\omega)}{|\Lambda|}  - 
\E_{\mu_\beta^+}[f_\Psi]\right|\geq u\right\} \leq
4\, \exp\left(-\frac{c_\varrho}{(2\opnorm{\Psi})^\varrho} |\Lambda|^{\frac{\varrho}{2}} u^{\varrho}\right),
\]
for all $u>0$
\end{itemize}
\end{theorem}

It is known that when $d=1$, one has $\E_{\mu_\beta^+}[s_0]=0$, whereas for $d=2$ 
(see {\em e.g.} \cite{CoyWu}) one has 
\[
\E_{\mu_\beta^+}[s_0]=\left(1-\big(\sinh(2\beta)\big)^{-4}\right)^{\frac{1}{8}}
\]
for all $\beta\geq \beta_c=\frac{1}{2}\ln(1+\sqrt{2})$. When $d\geq 3$ no explicit formula is known.

\begin{remark}\label{rem-ioffe}
\leavevmode\\
Probabilities of large deviations for the magnetization are well known for the Ising model.
At low temperature, one has ``surface-order'' large deviations, see \cite{schonmann} for instance. 
In particular one has the following estimate. Let $a,b$ such that
$-m_{\mu_\beta^+}<a<b<m_{\mu_\beta^+}$.
Then, the probability (under $\gmu_\beta^+$) that $M_n$ falls into $[a,b]$
is exponentially small in $(2n+1)^{d-1}$, as $n$ goes to infinity.
Comparing with Theorem \ref{thm:two-low-results}, we see that we get a weaker result  (since $\varrho d/2<d-1$ for all $d\geq 2$) which, however, is valid in any finite volume. Moreover, we get a bound not only for cubes but for all finite volumes.
\end{remark}

\subsection{Empirical pair correlations}

Let $f\in C^0(\Omega)$. For $\omega\in\Omega$, $x\in \Zd$ and $n\in\N$, define
\[
\Gamma_{n,x}(\omega)=
\frac{1}{(2n+1)^d} \sum_{y\in C_n} f(T_y\omega)f(T_{y+x}\omega).
\]
It follows from the multidimensional ergodic theorem (see, {\em e.g.}, \cite[p. 302]{Geo}) that, given an ergodic probability measure $\nu$, for each $x\in\Zd$,
\[
\Gamma_{n,x}(\omega)\xrightarrow[]{n\to+\infty}\E_\nu[ f\cdot f\circ T_x]
\]
for $\nu$-almost every $\omega$. Notice that $\E[\Gamma_{n,x}]=\E_\nu[ f\cdot f\circ T_x]$ for all
$n\in\N$ and for all $x\in\Zd$. We state a lemma whose proof follows the lines of Lemma \ref{lem:deltasums}.

\begin{lemma}\label{lem-correl}
Let $f\in\Delta_1(\Omega)$, $x\in\Zd$ and $n\in\N$. We have
\[
\Big\| \ushort{\gdelta}\Big(\sum_{y\in C_n} f\circ T_y (\cdot)f\circ T_{y+x}(\cdot)\Big)\Big\|_2^2
\leq 2(2n+1)^d \, \|f\|_\infty^2\, \|\ushort{\delta}(f)\|_1^2.
\]
\end{lemma}

\begin{proof}
For any $z\in\Zd$ we have
\begin{align*}
& \delta_{z}\left(\sum_{y\in C_n} f\circ T_y (\cdot)f\circ T_{y+x}(\cdot)\right)\\
&= \sup_{\substack{\omega,\omega' \\ \textup{differing only at}\, z}}
\left| 
\sum_{y\in C_n} \big( f\circ T_y (\omega)- f\circ T_y (\omega')\big)f\circ T_{y+x}(\omega)\right. \\
& \left. +
\sum_{y\in C_n} \big( f\circ T_{y+x} (\omega)- f\circ T_{y+x} (\omega')\big)f\circ T_{y}(\omega')
\right| \\
& \leq \|f\|_\infty \, \sum_{y\in C_n} \delta_z \big(f\circ T_y\big)
+  \|f\|_\infty \, \sum_{y\in C_n} \delta_z \big(f\circ T_{y+x}\big)\,.
\end{align*}
The finish the proof we use Young's inequality as in the proof of Lemma \ref{lem:deltasums}.
\end{proof}

We have the following results.

\begin{theorem}\label{prop:magnet}
\leavevmode\\
Let $\bgamma^{\Phi}$ be the specification of a potential $\Phi\in\mathscr{B}_T$ satisfying
Dobrushin's uniqueness condition \eqref{def:dobcoef}. Let $f\in\Delta_1(\Omega)$. Then
\begin{align*}
\MoveEqLeft[4] \gmu_\Phi\left\{
\omega\in \Omega : \left|\frac{\Gamma_{n,x}(\omega)}{(2n+1)^d} -  
\E_{\mu_\Phi}[ f\cdot f\circ T_x]
\right| \geq u  \right\}\\
& \leq 2 \exp\left(-\frac{ (1-\mathfrak{c}(\bgamma^{\Phi}))^2}{4 \|f\|_\infty^2 \|\ushort{\delta}(f)\|_1^2}\,
(2n+1)^d\, u^2\right)
\end{align*}
for all $u>0$, for all $n\in\N$ and for all $x\in\Zd$.
\end{theorem}

\begin{proof}
We apply Theorem \ref{thm:gemb-dob} and Lemma \ref{lem-correl} and replace $u$ by $(2n+1)^d u$.
\end{proof}

We can apply the previous theorem to $s_0(\omega)=\omega_0$ to get
\begin{align*}
\MoveEqLeft[4] \gmu_\Phi\left\{
\omega\in \Omega : \left|\frac{\Gamma_{n,x}(\omega)}{(2n+1)^d} -  
\E_{\mu_\Phi}[ s_0\cdot s_0\circ T_x]
\right| \geq u  \right\}\\
& \leq 2 \exp\left(-\frac{ (1-\mathfrak{c}(\bgamma^{\Phi}))^2}{16}\,
(2n+1)^d\, u^2\right)
\end{align*}
for all $u>0$, for all $n\in\N$ and for all $x\in\Zd$.

For the low-temperature Ising model, we have the following estimate.

\begin{theorem}\label{prop:magnet}
\leavevmode\\
Let $\gmu_\beta^+$ be the plus phase of the low-temperature ferromagnetic Ising model.
Let $f\in\Delta_1(\Omega)$.
Then there exists $\bar{\beta}>0$ such that, for each $\beta>\bar{\beta}$, there
exist $\varrho=\varrho(\beta)\in(0,1)$ such that
\begin{align*}
\MoveEqLeft[4] \gmu_\beta^+
\left\{\omega\in \Omega :  \left|\frac{\Gamma_{n,x}(\omega)}{(2n+1)^d} -  
\E_{\mu_\Phi}[ f\cdot f\circ T_x]
\right| \geq u\right\} \\
& \leq 4\, \exp\left(-\frac{c_\varrho}{(\sqrt{2}\|f\|_\infty \|\ushort{\delta}(f)\|_1)^\varrho}
\, (2n+1)^{\frac{\varrho d}{2}} u^{\varrho}\right),
\end{align*}
for all $u>0$, for all $n\in\N$ and for all $x\in\Zd$, where $c_\varrho>0$ is as in Theorem \ref{thm-seb}.
\end{theorem}

\begin{proof}
We apply Theorem \ref{thm-seb} and Lemma \ref{lem-correl}.
\end{proof}

\section{Application 2: Speed of convergence of the empirical measure}\label{subsec:scem}

\subsection{Generalities}

For $\Lambda\in\mathcal{P}$ and $\omega\in\Omega$, let
\[
\mathcal{E}_\Lambda(\omega)=\frac{1}{|\Lambda|} \sum_{x\in \Lambda} \gdelta_{T_x\omega}.
\]
Let $\nu$ be an ergodic measure on $(\Omega,\mathfrak{B})$. It is a consequence of the multidimensional ergodic theorem that, for any van Hove sequence $(\Lambda_n)_n$, we have
\[
\mathcal{E}_{\Lambda_n}(\omega)\xrightarrow[\text{weakly}]{n\to\infty} \nu
\]
for $\nu$-almost every $\omega\in \Omega$ (see \cite{tempelman}).
To quantify the speed of this convergence, we endow the set of probability measures on $\Omega$
with the Kantorovich distance $\distk$ defined by 
\begin{equation}\label{def-dK}
\distk(\mu_1,\mu_2)=
\sup_{\substack{G:\Omega\to\R\\ G\;1-\textup{Lipshitz}}}\left( \E_{\mu_1}[G] -\E_{\mu_2}[G]\right).
\end{equation}
A function $G:\Omega\to\R$ is $1$-Lipschitz if $|G(\omega)-G(\eta)|\leq \dist(\omega,\eta)$ where the distance
$\dist(\cdot,\cdot)$ is defined in \eqref{def:dist}. The distance $\distk$ metrizes the weak topology on the space of probability measures on $\Omega$.

We are interested in bounding the fluctuations of $\distk(\mathcal{E}_\Lambda(\omega),\mu)$ where $\mu$ will be a Gibbs measure. We start with a lemma.

\begin{lemma}\label{pizza}
Let $\nu$ be a probability measure.
For each $\Lambda\in\mathcal{P}$, consider the function
\[
F(\omega)=\sup_{\substack{G:\, \Omega\to\R\\ G\;1-\textup{Lipshitz}}}
\Big( \sum_{x\in \Lambda} G(T_x\omega)-\E_{\nu}[G]\Big).
\]
Then, we have
\begin{equation}\label{pouic}
\|\ushort{\delta}(F)\|_2^2\leq c_d\, |\Lambda|,
\end{equation} 
where $c_d>0$ is a constant only depending on $d$ (the dimension of the lattice).
\end{lemma}

\begin{proof}
Let $\omega, \omega'\in\Omega$ and $G:\Omega\to\R$ be a $1$-Lipschitz function. 
Without loss of generality, we can assume that $\E_{\nu}[G]=0$. We have
\[
\sum_{x\in \Lambda} G(T_x\omega) \leq \sum_{x\in \Lambda} G(T_x\omega')+
\sum_{x\in \Lambda}  d(T_x \omega,T_x \omega').
\]
Taking the supremum over $1$-Lipschitz functions thus gives
\[
F(\omega)-F(\omega')\leq \sum_{x\in \Lambda}  d(T_x \omega,T_x \omega').
\]
We can interchange $\omega$ and $\omega'$ in this inequality, whence
\[
|F(\omega)-F(\omega')|\leq \sum_{x\in \Lambda}  d(T_x \omega,T_x \omega').
\]
Now we assume that there exists $z\in\Zd$ such that $\omega_y=\omega'_y$ for all $y\neq z$. 
This means that $d(T_x \omega,T_x \omega')\leq  2^{-\|z-x\|_\infty}$ for all $x\in\Zd$, whence
\[
\delta_z(F)\leq \sum_{x\in \Lambda}  2^{-\|z-x\|_\infty}.
\]
Therefore, using Young's inequality as in the proof of Lemma \ref{lem:deltasums},
\begin{align*}
\|\ushort{\delta}(F)\|_2^2
& \leq \sum_{z\in\Zd} \left(\, \sum_{x\in\Zd} \un_{\Lambda}(x)\, 2^{-\|z-x\|_\infty}\right)^2\\
& \leq \sum_{x\in\Zd}  \un_{\Lambda}(x)\times \left(\, \sum_{z\in\Zd}  2^{-\|z\|_\infty}\right)^2.
\end{align*}
We thus obtain the desired estimate with $c_d=\Big(\sum_{z\in\Zd}  2^{-\|z\|_\infty}\Big)^2$.
\end{proof}

\subsection{Concentration of the Kantorovich distance}

We can now formulate two results.

\begin{theorem}[]\label{thm:d-Ed}
\leavevmode\\ 
Let $\Phi\in\mathscr{B}_T$ and assume that the associated specification $\bgamma^{\Phi}$ satisfies 
Dobrushin's uniqueness condition \eqref{def:dobcoef}. Denote by $\mu_\Phi$ the corresponding Gibbs measure. Then 
\begin{align*}
\MoveEqLeft \gmu_\Phi\!\left\{\omega\in \Omega\! : \!\Big|\distk(\mathcal{E}_\Lambda(\omega),\mu_\Phi )-
\E_{\mu_\Phi}\big[ \distk(\mathcal{E}_\Lambda(\cdot),\mu_\Phi)\big]\Big|\geq u  \right\}\\
& \leq 2\, \exp\big(-c\, |\Lambda| u^2\big)
\end{align*}
for all $\Lambda\in\mathcal{P}$ and for all $u>0$, where
\[
c=\frac{ (1-\mathfrak{c}(\bgamma^{\Phi}))^2}{2c_d}
\]
and $c_d$ is the constant appearing in Lemma \ref{pizza}.
\end{theorem}

\begin{proof}
We apply Theorem \ref{thm:gemb-dob} and the estimate \eqref{pouic} to get the announced inequality.
\end{proof}

For the plus-phase of the low temperature Ising model we can apply Theorem \ref{thm-seb} to
get immediately the following inequality.
\begin{theorem}[]\leavevmode\\ 
Let $\gmu_\beta^+$ be the plus phase of the low-temperature Ising model. There exists  $\bar{\beta}$ such that, for each $\beta>\bar{\beta}$, there exist $\varrho=\varrho(\beta)\in(0,1)$ and a constant $c_\varrho>0$ such that 
\begin{align*}
\MoveEqLeft \gmu_\beta^+\left\{\omega\in \Omega : 
\Big|\distk(\mathcal{E}_\Lambda(\omega),\mu_\beta^+ )-
 \E_{\mu_\beta^+}\big[\distk(\mathcal{E}_\Lambda(\cdot),\mu_\beta^+)\big]\Big|\geq u\right\} \\
 & 
\leq 4\, \exp\left(-c_\varrho |\Lambda|^{\frac{\varrho}{2}}u^{\varrho}\right)
\end{align*}
for all $\Lambda\in\mathcal{P}$ and for all $u>0$.
\end{theorem}

\begin{proof}
It is a direct application of Theorem \ref{thm-seb} and estimate \ref{pouic}.
\end{proof}

\subsection{Expectation of Kantorovich distance}

At this stage we can only control $\distk(\mathcal{E}_\Lambda(\omega),\mu_\Phi )$ minus its expected value. 
So we still need to obtain an upper bound for $\E_{\mu_\Phi}\big[ \distk(\mathcal{E}_\Lambda(\cdot),\mu_\Phi)\big]$.
For the sake of simplicity, we will provide an asymptotic upper bound in the cardinality of $\Lambda$. The reader can infer from the proofs that giving a non-asymptotic upper bound for all $\Lambda$ is possible but tedious.

Let $\nu$ be a probability measure on $(\Omega,\mathfrak{B})$, $f:\Omega\to\R$ a continuous function
and $\Lambda$ a finite subset of $\Zd$. Define
\begin{equation}\label{bolastro}
X^\Lambda_f= \frac{1}{|\Lambda|}\sum_{x\in\Lambda} \left( f\circ T_x-\E_\nu[f]\right)
\end{equation}
We have
\[
\sup_{f\in\mathscr{F}} X^\Lambda_f
=
\distk\left(\mathcal{E}_\Lambda(\cdot),\nu\right)
\]
where $\mathscr{F}$ is the collection of all Lipschitz functions $f:\Omega\to\R$ with Lipschitz constant less than or equal to one.
We want to estimate the expected distance
\[
\E_\nu \left[\distk\left(\mathcal{E}_\Lambda(\cdot),\nu\right)\right]=
\E_\nu \left[\sup_{f\in\mathscr{F}} X^\Lambda_f\right].
\]
Notice that we can subtract a constant from $f$ without influencing $X^\Lambda_f$, therefore, using that
$f$ is Lipschitz and the maximal distance between two configurations in $\Omega$ is equal to $1$, we
can assume, without loss of generality, that the functions in $\mathscr{F}$ take values in $[0,1]$.
Estimating such a supremum is a classical problem. We adapt the line of thought of \cite{vH} to our context where we have to do some extra, non-trivial, work, see Remark \ref{rem:KT} below for more details.

\subsubsection{Case 1: Gaussian concentration bound case}

Let $\epsilon>0$ be given.
We want to find a finite collection of functions $\mathscr{F}_\epsilon$ such that the following two properties are satisfied:
\begin{enumerate}
\item {\bf $\epsilon$-net property.} For all $f\in\mathscr{F}$ there exists $g\in\mathscr{F}\!_\epsilon$ which is uniformly $\epsilon$ close to $f$, {\em i.e.}, such that $\|f- g\|_\infty\leq \epsilon$.
\item {\bf Uniform $\epsilon$-Gaussian upper bound property.} There exists $D'>0$ (possibly depending on $\Lambda$) such that for all $f\in\mathscr{F}_\epsilon$ and all $\lambda\in\R$ we have
\begin{equation}
\label{bibo}
\E_\nu \left[\exp\big(\lambda X^\Lambda_f\big)\right] \leq \exp\big(\lambda \epsilon\big) \exp\left(D' \lambda^2 \right).
\end{equation}
\end{enumerate}
Such a collection  $\mathscr{F}_\epsilon$ is called a good $\epsilon$-net for $\mathscr{F}$.
Let us now assume that such a $\mathscr{F}_\epsilon$ is given. Then we have
\begin{lemma}
For all $\mathscr{F}_\epsilon$ good we have the upper bound
\begin{equation}\label{forro}
\E_\nu\left[\sup_{f\in\mathscr{F}}X^\Lambda_f\right]
\leq 2\left(\epsilon + \sqrt{D'\log |\mathscr{F}_\epsilon|}\,\right).
\end{equation}
\end{lemma}
\begin{proof}
For any $\lambda>0$, we have, using Jensen's inequalty and \eqref{bibo},
\begin{align*}
\E_\nu\left[\sup_{f\in\mathscr{F}_\epsilon}X^\Lambda_f\right]
&= \frac1\lambda \E_\nu\left[\log \exp\big(\lambda\sup_{f\in\mathscr{F}_\epsilon}X^\Lambda_f\big)\right]\\
&\leq
\frac1\lambda \log\E_\nu\left[\exp\big(\lambda\sup_{f\in\mathscr{F}_\epsilon} X^\Lambda_f\big)\right]\\
&\leq
\frac1\lambda \log \E_\nu \left[\sum_{f\in\mathscr{F}_\epsilon} \exp\big(\lambda X^\Lambda_f\big)\right]\\
&\leq
\frac1\lambda \left(\log|\mathscr{F}_\epsilon| + \lambda^2 D' +\lambda\epsilon\right).
\end{align*}
Optimizing w.r.t. $\lambda$ gives
\[
\E_\nu\left[\, \sup_{f\in\mathscr{F}_\epsilon}X^\Lambda_f\right]\leq 2\sqrt{D'\log |\mathscr{F}_\epsilon|}+\epsilon.
\]
The statement of the lemma now follows from the $\epsilon$-net property of $\mathscr{F}_\epsilon$, {\em i.e.},
\[
\E_\nu\left[\, \sup_{f\in\mathscr{F}}X^\Lambda_f\right]\leq \epsilon +
\E_\nu\left[\sup_{f\in\mathscr{F}_\epsilon}X^\Lambda_f\right].
\]
\end{proof}
We now first show that if $\mathscr{F}_\epsilon$ is a finite collection of functions
which are all uniformly close to a $1$-Lipschitz function $f$, then \eqref{bibo} holds.
\begin{lemma}
\label{lipapproxlem}
If $g$ is such that there exist a $1$-Lipschitz function $f$ such that $\|f-g\|_\infty\leq \epsilon$, and if $\nu$ satisfies $\gcb{D}$, then, for all $\lambda\in\R$, one has
\[
\E_\nu\left[\exp\big(\lambda X^\Lambda_g\big)\right]\leq \exp(\lambda\epsilon)  \exp \left(D' \lambda^2\right).
\]
\end{lemma}
\begin{proof}
It suffices to show that for all $f$ $1$-Lipschitz we have
\[
\E_\nu\left[\exp\big(\lambda X^\Lambda_f\big)\right]\leq \exp \left(D' \lambda^2\right),
\]
where $D'$ does not depend on $f$. This is the consequence of the Gaussian concentration bound and the proof
of Lemma \ref{pizza}.
\end{proof}

From what precedes, we are left to find a good $\epsilon$-net $\mathscr{F}_\epsilon$ in our setting.
The first step is to find a $\epsilon$-net for the configuration space $\Omega$.
This is defined as a finite set of configurations $\Omega_\epsilon\subset\Omega$ such that
for all $\eta\in\Omega$ there exists $\zeta\in\Omega_\epsilon$ with $d(\eta,\zeta)\leq \epsilon$.
The following lemma gives such a net.
\begin{lemma}
\label{bimlem}
Let $\overline{\eta}$ be a fixed configuration in $\Omega$.
We define for $n\in\N$ the set
\[
\Omega_n^{\overline{\eta}}= \{ \eta\in\Omega: \eta_{C_n^c}={\overline{\eta}}_{C_n^c}\}.
\]
Then $\Omega_n^{\overline{\eta}}$ is a $2^{-n}$ net of cardinality $|S|^{|C_n|}$.
\end{lemma}
\begin{proof}
This follows  immediately from the definition of the distance in $\Omega$.
\end{proof}

If $f$ is a $1$-Lipschitz function, then we have that if $\eta_{C_n}=\zeta_{C_n}$, $|f(\eta)-f(\zeta)|\leq 2^{-n}$.
Notice that we can view $\Omega_n^{\overline{\eta}}$ in \eqref{bimlem} as a copy of $S^{C_n}$ via the map
\[
\psi: S^{C_n}\to \Omega_n^{\overline{\eta}}: \alpha\mapsto \alpha_{C_n}\overline{\eta}_{C_n^c}.
\]
This means that  ordering the elements of $\Omega_n^{\overline{\eta}}$ is the same as ordering the elements of $S^{C_n}$.
The aim now  is to order the elements of the net $\Omega_n^{\overline{\eta}}$ in such a way that the distances
between successive elements in the ordering are as small as possible. Because $\Omega$ is a totally disconnected
space, we will not be able  to avoid that in this order there are distances of
$2^{-(n-1)}, 2^{-(n-2)},\ldots, 2^{-1}$. The following lemma explains the hierarchical structure of the ordering.
\begin{lemma}
There exists an ordering of $S^{C_n}$ of the following type
\begin{align*}
& \alpha^0\\
&  \alpha^{1,1},\ldots,\alpha^{1,P(n,1)}\quad \textup{(first list)}\\
& \alpha^{2,1},\ldots,\alpha^{2,P(n,2)}\quad \textup{(second list)}\\
& \qquad\quad \vdots\\
& \alpha^{n,1},\ldots, \alpha^{n,P(n,n)}\quad \textup{(}n\textup{th list)}
\end{align*}
such that for all $k,\ell\in \{0,\ldots, n\}, i\in 1,\ldots, P(n,k), j\in 1,\ldots, P(n,\ell)$, we have
\[
d(\psi(\alpha^{k,i}), \psi(\alpha^{\ell,j})) = 2^{-(n- k\vee \ell)},
\]
where $d$ is the distance defined in \eqref{def:dist}. Here $P(n,1)=|S|^{|C_n\setminus C_{n-1}|}$, 
$P(n,2)=S^{|C_n\setminus C_{n-1}|+|C_{n-1}\setminus C_{n-2}|}$, etc.
\end{lemma}
\begin{proof}
We choose an arbitrary first element $\alpha^0$ in $S^{C_n}$.
The next elements form an arbitrary enumeration of the configurations
which are equal to $\alpha^0$ in $C_{n-1}$, but different in at least
one site $x\in C_{n}\setminus C_{n-1}$. There are at most $P(n,1)=|S|^{|C_n\setminus C_{n-1}|}$ such configurations.
They are all at distance $2^{-n}$ from $\alpha^0$ and from each other.
Next are the elements at distance $2^{-(n-1)}$ from $\alpha$.
These are at most $|S|^{|C_{n-1}\setminus C_{n-2}|}$ configurations associated to each configuration
in the previous list, hence in total this gives $P(n,2)=S^{|C_n\setminus C_{n-1}|+|C_{n-1}\setminus C_{n-2}|}$ configurations in the second list. And so on and so forth.
We go on like this, ``peeling'' off the cube $C_n$ by successive boundary layers $C_n\setminus C_{n-1}, C_{n-1}\setminus C_{n-2},\ldots ,\{0\}$, and end up with the configurations at distance $1/2$ from $\alpha^0$, of which  there are $P(n,n)=S^{|C_n|-1}$
\end{proof}
Now we want to make our $\epsilon$ net $\mathscr{F}_\epsilon$. We choose $n$ such that $2^{-n}\leq \epsilon\leq 2^{-(n-1)}$. We will give a function value to each $\psi(\alpha), \alpha\in S^{C_n}$, which will only depend on $\alpha$, so we identify it with a function $f: S^{C_n}\to\R$.
Because the functions will take values in $\left\{0,\frac{1}{2^n},\ldots, \frac{2^n-1}{2^n},1\right\}$, we have $2^n+1$ possibilities for the function value of $\alpha^0$.
Because we will choose the functions in $\mathscr{F}_\epsilon$ to be $1$-Lipschitz, this restricts the possible values
of the functions at $\alpha_1$. Indeed, given the function value of $\alpha^0$,
for the function values of the first list, which contains configurations which are at distance $2^{-n}\leq \epsilon $,
we have at most three possibilities, namely $f(\alpha^0)+s$ with $s\in \{-2^{-n},0,2^{-n}\}$.
Given the function values in the first list, all the elements of the second list are at distance $2^{-(n-2)}\leq 2\epsilon$ from $\alpha^0$ and from any element of the first list, so we have now $2^2+1$ possible function values, associated to any configuration of the second ``layer'', {\em i.e.}, $|S|^{|C_{n-1}\setminus C_{n-2}|}$ configurations, and so on and so forth. The number of functions we thus obtain is upper bounded by
\[
|\mathscr{F}_\epsilon|=(2^n+1) (2+1)^{P(n,1)} (2^2+1)^{P(n,2)} \cdots (2^n+1)^{P(n,n)}.
\]
Taking the logarithm of this expression gives, using the (crude) upperbound $\log(2^{n}+1)\leq n+1$
\[
\log|\mathscr{F}_\epsilon|
\leq
n+1 + P(n,1)2 + P(n,2) 3 +\cdots +(n+1)P(n,n)=:\mathscr{K}_\epsilon.
\]
It is clear that the asymptotic behavior of this expression
is dominated by the last term, {\em i.e}, we have
\[
\mathscr{K}_\epsilon\sim |S|^{(2\log(1/\epsilon)+1)^d} (\log(\tfrac1\epsilon)+1),
\]
where $a_\epsilon \sim b_\epsilon$ means that $a_\epsilon/b_\epsilon \to 1$ as $\epsilon$ goes
to $0$.

\begin{remark}\label{rem:KT}
Let us stress that we cannot obtain the previous estimate by a direct application of the standard
results on $\epsilon$-entropy. To be more specific, our estimate does not follow from Theorem XV in \cite{KT} for the totally disconnected metric space $\Omega$. The problem stems from the fact that we cannot metrically embed $\Omega$ into a finite-dimensional parallelepiped, except in dimension $d=1$.
\end{remark}

We now analyse how the bound \eqref{forro} behaves.
By lemma \ref{pizza} we have that the constant $D$ in this bound is of the form
$D'=D_1/|\Lambda|$, where $D_1$ is independent of $|\Lambda|$. Our aim now is to
extract the leading order behavior in $|\Lambda|$ of the optimal found in \eqref{forro}, where
we replace $ \log |\mathscr{F}_\epsilon|$ by $\mathscr{K}_\epsilon$.
I.e., we compute
\begin{equation}\label{upperboend}
B(|\Lambda|)=2 \inf_{\epsilon >0}\left(\epsilon + \sqrt{ \frac{D_1 \mathscr{K}_\epsilon}{|\Lambda|}}\, \right).
\end{equation}
Let us abbreviate $\log(1/\epsilon)=v(\epsilon)$.
The optimal $\epsilon=\epsilon^*$ is the solution of
\[
\frac12 \log|\Lambda| +\frac12\log (1/D_1)
= v(\epsilon) +\frac12 \log(|S|) (2 v(\epsilon))^d+\chi(\epsilon),
\]
where $\chi(\epsilon)$ is of lower order as $\epsilon$ goes to $0$.
In order to collect the leading order behavior of $B(|\Lambda|)$ in $|\Lambda|$ on the logarithmic scale, we will therefore omit $\chi(\epsilon)$ in this equation, which will lead to lower order factors in the asymptotic behavior of $B(|\Lambda|)$. We will also omit the term $\frac12\log (1/D_1)$ for the same reason.\\
Let us now introduce two notions of asymptotic comparison. For two strictly positive sequences $(a_n)$ and $(b_n)$, we write $a_n\asymp b_n$ if $\frac{\log a_n}{\log b_n}\to 1$ as $n\to\infty$, and $a_n \preceq b_n$ if
$\limsup_n \frac{\log a_n}{\log b_n}\leq 1$. For instance we have $4n^{-1/2}\log(n)\log(\log(n))\asymp n^{-1/2}$, and
$n^{2n} e^{n} \preceq n^{3n}$. Similarly, for two sequences $(a_\Lambda)$ and $(b_\Lambda)$ indexed by finite
subsets of $\Zd$ we denote $a_\Lambda \asymp b_\Lambda$ if, for every sequence $(\Lambda_n)$ such that $|\Lambda_n|\to +\infty$ as $n\to+\infty$, we have $\frac{\log a_{\Lambda_n}}{\log b_{\Lambda_n}}\to 1$. Analogously,
we define $a_\Lambda \preceq b_\Lambda$. 
\\
As a consequence, for $\epsilon=\epsilon^*$ we find that both terms in the rhs of \eqref{upperboend} are of the same order, and hence on this level of roughness, the behavior of $B(|\Lambda|)$ is the same as that of $\epsilon^*$.

Proceeding like this, we find the following leading order behavior of $B(|\Lambda|)$ as a function of the dimension.
\begin{enumerate}
\item Dimension $d=1$.
\[
\epsilon^*\asymp |\Lambda|^{-\frac{1}{2}(1+\log |S|)^{-1}}.
\]
\item Dimension $d\geq 2$.
\[
\epsilon^*\asymp \exp\left(-\frac12 \left(\frac{\log|\Lambda|}{\log|S|}\right)^{1/d}\right).
\]
Notice that this does not give the previous bound when we plug in $d=1$ because (only) for $d=1$
the additional tern $v(\epsilon)$ is of the same order as the second term  $\frac12 (\log|S|) (2 v(\epsilon)+1)^d$.
\end{enumerate}
As a conclusion we obtain the following asymptotic estimates.
\begin{theorem}\label{thm:bKd}
Let $\nu$ be a probability measure on $\Omega$ satisfying $\gcb{D}$. Then
\[
\E_\nu \left[\distk\left(\mathcal{E}_\Lambda(\cdot),\nu\right)\right]
\preceq
\begin{cases}
|\Lambda|^{-\frac{1}{2}(1+\log |S|)^{-1}} & \text{if}\quad d=1\\
\exp\left(-\frac12 \left(\frac{\log|\Lambda|}{\log|S|}\right)^{1/d}\right) & \text{if}\quad d\geq 2.
\end{cases}
\]
\end{theorem}

\subsubsection{Case 2: Moment concentration bound case}

Let us now see what can be done when exponential moments do not exist, {\em i.e.}, if we do  not have GCB.
We call then an $\epsilon$-net $\mathscr{F}_\epsilon$ good if  we have
\begin{enumerate}
\item {\bf The $\epsilon$-net property.} For  all $g\in\mathscr{F}$ there exists $f\in \mathscr{F}_\epsilon$ such that
$\|f- g\|_\infty\leq \epsilon$.
\item {\bf The $\epsilon$-Moment bound.} For all $f\in\mathscr{F}_\epsilon$
\begin{equation}
\| X^\Lambda_f\|_{L^{2p}(\nu)}\leq \epsilon + \frac{C_{2p}^{1/2p} }{\sqrt{|\Lambda|}}.
\end{equation}
\end{enumerate}
Then going through the same reasoning, as before (but with the function $x\mapsto e^{\lambda x}$ replaced by
$x\mapsto |x|^{2p}$) we obtain the estimate
\begin{equation}\label{strumpf}
\E_\nu\left[\,\sup_{f\in\mathscr{F}} X^\Lambda_f \right]
\leq \epsilon +
|\mathscr{F}_\epsilon|^{1/2p} \left(\epsilon + \frac{C_{2p}^{1/2p}}{\sqrt{|\Lambda|}}\right).
\end{equation}
As in the previous subsection, we have $|\mathscr{F}_\epsilon|\asymp \exp (\exp(\alpha (\log(1/\epsilon)^d)))$, with
$\alpha=2^d \log|S|$.
Let us furthermore assume that we have the bound
\begin{equation}\label{pouicpouc}
C_{2p}\leq p^{\kappa 2p}
\end{equation}
for some $\kappa\geq 1/2$. In particular, for the low-temperature Ising model, we have $\kappa=1$ (see Remark \ref{rem-moments-lowT}), whereas we have $\kappa=1/2$ in the case of a Gaussian concentration bound.
Then we analyse as before, {\em i.e.}, on the level of logarithmic equivalence, the bounds we obtain from \eqref{strumpf}.
\begin{enumerate}
\item Dimension $d=1$. Then we have $|\mathscr{F}_\epsilon|\asymp \exp (\epsilon^{-\alpha})$. We find the upperbound
\[
B(|\Lambda|)\preceq |\Lambda|^{-\frac{1}{2(\alpha \kappa +1)}}
\]
\item Dimension $d\geq 2$ we find
\[
B(|\Lambda|) \preceq \exp \left(-\left(\frac{\log|\Lambda|}{2\alpha\kappa}\right)^{1/d}\right)
\]
\end{enumerate}
As a conclusion we obtain the following asymptotic estimates.
\begin{theorem}\label{thm:bKdbis}
Let $\nu$ be a probability measure on $(\Omega,\mathfrak{B})$ satisfying $\mcb{2p,C_{2p}}$ for all $p\in\N$. 
Moreover assume that \eqref{pouicpouc} holds.
Then
\[
\E_\nu \left[\distk\left(\mathcal{E}_\Lambda(\cdot),\nu\right)\right]
\preceq
\begin{cases}
|\Lambda|^{-\frac{1}{2(\alpha \kappa +1)}} & \text{if}\quad d=1\\
\exp \left(-\left(\frac{\log|\Lambda|}{2\alpha\kappa}\right)^{1/d}\right) & \text{if}\quad d\geq 2.
\end{cases}
\]
\end{theorem}

Notice that when $\kappa=1/2$, this theorem is exactly the bound we obtained in Theorem \ref{thm:bKd}.

\section{Application 3: Fluctuations in the Shannon-McMillan-Breiman theorem and its analog for relative entropy}\label{sec:SMB}

If $\nu$ is an ergodic probability measure, the following holds:
\[
\lim_{n\to\infty}-\frac{\log\nu(\EuScript{C}_n(\omega))}{(2n+1)^d}=h(\nu)\quad\text{for}\;\nu\text{-a.e.} \;\omega.
\]
This is usually referred to as the Shannon-Millan-Breiman theorem for random fields and was proved in \cite{kieffer}.
If $\Phi\in\mathscr{B}_T$ then we have
\[
\lim_{n\to\infty}-\frac{1}{(2n+1)^d}\log\frac{\nu(\EuScript{C}_n(\omega))}
{\mu_\Phi(\EuScript{C}_n(\omega))}=h(\nu|\mu_\Phi)\quad\text{for}\;\nu\text{-a.e.} \;\omega
\]
where $\mu_\Phi$ is any shift-invariant Gibbs measure associated with $\Phi$, and where $h(\nu|\mu_\Phi)$
is the relative entropy (per site) of $\nu$ with respect to $\mu_\Phi$ (cf. \eqref{def:relent}, \eqref{formule-rel-entropy}). This result can be deduced using the Shannon-Millan-Breiman theorem, \eqref{eq:control-cyl}, and the multidimensional ergodic theorem \cite{tempelman} applied to the measure $\nu$. 
Our goal is to control the fluctuations of both quantities around their respective limits when $\nu$ is a Gibbs measure. We have the following results.

\begin{theorem}\label{thm:SMB}
\leavevmode\\ 
Let $\Phi\in\mathscr{B}_T$ be a potential whose specification $\bgamma^{\Phi}$ satisfies 
Dobrushin's uniqueness condition \eqref{def:dobcoef}. Then there exists $u_0>0$ such that
\begin{align*}
& \gmu_\Phi\!\left\{\!\omega\in \Omega :
\left|\frac{-\log\gmu_\Phi(\EuScript{C}_n(\omega))}{(2n+1)^d}-
\E_{\mu_\Phi}\!\!\left[\frac{-\log\mu_\Phi(\EuScript{C}_n(\cdot))}{(2n+1)^d}\right]\right| \geq \frac{u}{(2n+1)^{\frac{d-1}{2}}}\! \right\}\\
&  \leq 2\, \exp\left(- \frac{ (1-\mathfrak{c}(\bgamma^{\Phi}))^2}{8\opnorm{\Phi}^2}\, (2n+1)\, u^2\right)
\end{align*}
for all $n\in\N$ and for all $u\geq u_0$. Suppose, in addition to Dobrushin's uniqueness condition, that 
\eqref{cond-decay} holds, then there exists $u_0>0$ such that
\begin{align*}
\MoveEqLeft \gmu_\Phi\left\{\omega\in \Omega :
\Big|\frac{-\log\gmu_\Phi(\EuScript{C}_n(\omega))}{(2n+1)^d}-h(\mu_\Phi)\Big| \geq \frac{u}{(2n+1)^{p(d)}} \right\}\\
& \leq 2\, \exp\left(-\frac{ (1-\mathfrak{c}(\bgamma^{\Phi}))^2}{32\opnorm{\Phi}^2}\, (2n+1)\, u^2\right)
\end{align*}
for all $n\in\N$ and for all $u\geq u_0$, where
\[
p(d)=
\begin{cases}
\frac12 & \textup{if}\; d=2\\
1 & \textup{if}\; d\geq 3.
\end{cases}
\]
\end{theorem}

\begin{proof}
For each $n\in\N$, the function $\omega\mapsto F(\omega)=-\log\mu_\Phi(\EuScript{C}_n(\omega))$ is a local function
(with dependence set $C_n$). We apply \eqref{eq:dlr} with $A=\EuScript{C}_n(\omega)$ and $\Lambda=C_n$ which gives
\begin{equation}\label{eq:mes-cyl}
\mu_\Phi(\EuScript{C}_n(\omega))=\int \dd \mu_\Phi(\eta) \, \bgamma^{\Phi}_{C_n}(\omega|\eta).
\end{equation}
Let $x\in C_n$, and $\omega,\tilde{\omega}\in\Omega$ such that $\omega_y=\tilde{\omega}_y$ for
all $y\neq x$. We want to control 
\[
\log \frac{\mu_\Phi(\EuScript{C}_n(\omega))}{\mu_\Phi(\EuScript{C}_n(\tilde{\omega}))}
\]
Using \eqref{eq:mes-cyl}, \eqref{eq:spe} and \eqref{eq:bruit} we obtain
\[
\frac{\mu_\Phi(\EuScript{C}_n(\omega))}{\mu_\Phi(\EuScript{C}_n(\tilde{\omega}))}=
\frac
{
\mathlarger{\int} \dd \mu_\Phi(\eta) \,
\mathlarger{\frac{\bgamma^{\Phi}_{C_n}(\omega|\eta)}{\bgamma^{\Phi}_{C_n}(\tilde{\omega}|\eta)}} \, \bgamma^{\Phi}_{C_n}(\tilde{\omega}|\eta)
}
{
\mathlarger{\int} \dd \mu_\Phi(\eta) \, \bgamma^{\Phi}_{C_n}(\tilde{\omega}|\eta)
}
\leq \exp(2 \opnorm{\Phi}).
\]
Hence
\[
\left|\log \frac{\mu_\Phi(\EuScript{C}_n(\omega))}{\mu_\Phi(\EuScript{C}_n(\tilde{\omega}))}\right| \leq 
2\opnorm{\Phi}
\]
which immediately implies that
\[
\delta_x(F)\leq 2\opnorm{\Phi}
\]
for all $x\in C_n$ ($\delta_x(F)=0$ for all $x\in\Zd\backslash C_n$). 
The first statement then follows at once by applying Theorem \ref{thm:gemb-dob} and rescaling $u$. 
If one can control the measure of cylinders as in \eqref{eq:control-cyl}, we can obtain a good estimate for the expectation of $-\log\mu_\Phi(\EuScript{C}_n(\omega))$ and get the second statement. 
Since $\Phi$ satisfies \eqref{cond-decay} we have \eqref{eq:control-cyl}, hence we obtain
\begin{equation}\label{esp-SMB}
\left|\frac{1}{(2n+1)^d} \, \E_{\mu_\Phi}[-\log\mu_\Phi(\EuScript{C}_n(\omega))]-h(\gmu_\Phi)\right|
\leq \frac{C_\Phi}{n}
\end{equation}
where we used the variational principle \eqref{VP}. Notice that the bound is independent of $d$.
The announced inequalities follow with $u_0=C_\Phi$.
\end{proof}

Following the same train of thought as in the previous theorem, we obtain the companion result for relative entropy.

\begin{theorem}\label{thm:rel-ent}
\leavevmode\\ 
Let $\Phi\in\mathscr{B}_T$ be a potential whose specification $\bgamma^{\Phi}$ satisfies 
Dobrushin's uniqueness condition \eqref{def:dobcoef}, and let $\Psi\in\mathscr{B}_T$ satisfying \eqref{cond-decay}. 
Let $\mu_\Psi$ be any shift-invariant Gibbs measure associated with $\Psi$.
Then there exists $u_0>0$ such that
\begin{align*}
& \gmu_\Phi \bigg\{\omega\in \Omega :
\left|\frac{1}{(2n+1)^d}\log\frac{\gmu_\Phi(\EuScript{C}_n(\omega))}{\gmu_\Psi(\EuScript{C}_n(\omega))}
- \E_{\mu_\Phi}\!\!\left[\frac{1}{(2n+1)^d}\log\frac{\gmu_\Phi(\EuScript{C}_n(\omega))}{\gmu_\Psi(\EuScript{C}_n(\omega))}\right]\right| \\
&   \hspace{10.5cm} \geq \frac{u}{(2n+1)^{\frac{d-1}{2}}}\! \bigg\}\\
&  \leq 2\, \exp\left(- \frac{ (1-\mathfrak{c}(\bgamma^{\Phi}))^2}{8\big(\opnorm{\Phi}+\opnorm{\Psi}\big)^2}\, (2n+1)\, u^2\right)
\end{align*}
for all $n\in\N$ and for all $u\geq u_0$. Suppose, in addition to Dobrushin's uniqueness condition, that 
\eqref{cond-decay} holds for $\Phi$. Then there exists $u_0>0$ such that
\begin{align*}
\MoveEqLeft \gmu_\Phi\left\{\omega\in \Omega :
\left|\frac{1}{(2n+1)^d}\log\frac{\gmu_\Phi(\EuScript{C}_n(\omega))}{\gmu_\Psi(\EuScript{C}_n(\omega))}
-h(\mu_\Phi|\mu_\Psi)
\right| \geq \frac{u}{(2n+1)^{p(d)}} \right\}\\
& \leq 2\, \exp\left(-\frac{ (1-\mathfrak{c}(\bgamma^{\Phi}))^2}{32\big(\opnorm{\Phi}+\opnorm{\Psi}\big)^2}\, (2n+1)\, u^2\right)
\end{align*}
for all $n\in\N$ and for all $u\geq u_0$, where $p(d)$ is defined as in the previous theorem.
\end{theorem}

We now formulate a companion result on the Ising ferromagnet at low temperature. It is a simple consequence of
Theorem \ref{thm-seb} and inequality \eqref{esp-SMB}. 

\begin{theorem}
\leavevmode\\ 
Let $\gmu_\beta^+$ be the plus phase of the low-temperature Ising model on the lattice $\Zd$, $d\geq 2$. 
There exist two constants, $u_0=u_0(d)>0$ and $\bar{\beta}=\bar{\beta}(d)>0$ such that, for each $\beta>\bar{\beta}$, there exist $\varrho=\varrho(\beta)\in(0,1)$ and $\tilde{c}_\varrho>0$ such that the following two estimates hold:
\begin{itemize}
\item[\textup{(a)}]
If $d=2$ we have
\begin{align*}
\MoveEqLeft[2] \gmu_\beta^+\left\{\omega\in \Omega : 
\left|\frac{-\log\gmu_\beta^+(\EuScript{C}_n(\omega))}{(2n+1)^2}-h(\gmu_\beta^+)\right| \geq \frac{u}{n^\tau} \right\}  \\
& \leq 4\, \exp\left(-\tilde{c}_\varrho (2n+1)^{\varrho(1-\tau)}u^{\varrho}\right)
\end{align*}
for all $n\in\N$, for all $u\geq u_0$ and for any $\tau\in (0,1)$, where $\tilde{c}_\varrho = c_{\varrho}\, 2^{-\frac{5}{2}\varrho} \beta^-{\frac{\varrho}{2}}$.
\item[\textup{(b)}]
If $d\geq 3$, we have   
\begin{align*}
\MoveEqLeft[2] \gmu_\beta^+\left\{\omega\in \Omega : 
\left|\frac{-\log\gmu_\beta^+(\EuScript{C}_n(\omega))}{(2n+1)^d}-h(\gmu_\beta^+)\right| \geq \frac{u}{n} \right\}  \\
& \leq 4\, \exp\left(-\tilde{c}_\varrho (2n+1)^{\varrho(\frac{d}{2}-\tau)}u^{\varrho}\right)
\end{align*}
for all $n\in\N$, for all $u\geq u_0$ and for any $1<\tau<\frac{d}{2}$, where
$\tilde{c}_\varrho = c_{\varrho}\, 2^{-2\varrho}(d\beta)^{-\frac{\varrho}{2}}$.
\end{itemize}
In both cases, $c_{\varrho}=c_{\varrho}(d)$ is the constant appearing in Theorem \ref{thm-seb}.
\end{theorem}

The reader can now infer the counterpart of Theorem \ref{thm:rel-ent} for the low-temperature Ising model. 

\section{Application 4: First occurrence of a pattern of a configuration in another configuration}\label{sec:W}

For a subset $\Lambda$ of $\Zd$, we refer to an element $a=(a_x,x\in\Lambda)\in S^{\Lambda}$ as a pattern supported by $\Lambda$. 
Given $x\in\Zd$, we say that the patterns $a \in S^{\Lambda}$ and $b\in S^{\Lambda+x}$ are congruent
if $a_y=b_{y+x}$ for every $y\in\Lambda$. Now, let $\eta,\omega\in\Omega$. For each $n\in\N$, we look for the smallest hypercube $C_k$ such that ``$\eta_{C_n}$ appears in $\omega_{C_k}$''. This means that there is a pattern $a$ whose support lies inside $C_k$ such that  $\eta_{C_n}$ and $a$ are congruent, and that, if we take $k'<k$, there is no pattern whose support lies inside $C_{k'}$ which is congruent to $\eta_{C_n}$. This event can be seen as the first occurrence of the pattern $\eta_{C_n}$ in the configuration $\omega$: imagine that we are increasing at a constant rate the `window' $C_k$ in $\omega$ until we observe the pattern $\eta_{C_n}$ for the first time.

We denote by $W_n(\eta,\omega)$ the cardinality of the random hypercube $C_k$ we have just defined. It turns out that the natural random variable to consider is $\log W_n(\eta,\omega)$. Indeed, one can prove (see \cite{ACRV}) that if $\Phi$ of finite range and  $\bgamma^{\Phi}$ satisfies Dobrushin's uniqueness condition and $\nu$ is any ergodic measure, then 
\[
\lim_{n\to \infty} \frac{1}{(2n+1)^d} \log W_n(\eta,\omega)= h(\nu)+h(\nu|\mu_\Phi),\; \nu\otimes\mu_\Phi-\text{a.e.}.
\]
Now, fix $n$ and $\eta$. It is quite obvious that no a priori control will be possible on $|\log W_n(\eta,\omega)-\log W_n(\eta,\omega')|$ for {\em all} configurations $\omega,\omega'$ which differ only at a site $x$. Indeed, changing $\omega$ in a single site can cause an arbitrary increase of the size of the hypercube in which we will see $\eta_{C_n}$. This is because we have to consider the worst case changes, not only typical changes for which things would go well. Nevertheless, we will obtain concentration inequalities by making a detour.

\begin{theorem}\label{thm:patterns}
\leavevmode\\  
Assume that $\Phi$ is of finite range and the associated specification $\bgamma^{\Phi}$ satisfies Dobrushin's uniqueness condition \eqref{def:dobcoef}. Let $\Psi$ be a potential satisfying \eqref{cond-decay} and such
that its specification satisfies Dobrushin's uniqueness condition.
When $\Phi\neq \Psi$, let
\[
c_{\Phi,\Psi}=\frac{ (1-\mathfrak{c}(\bgamma^{\Phi}))^2}{128\,\big( \opnorm{\Phi}+\opnorm{\Psi}\big)^2},
\]
and 
\[
c_{\Phi,\Phi}=c_\Phi=\frac{ (1-\mathfrak{c}(\bgamma^{\Phi}))^2}{128\,\opnorm{\Phi}^2}.
\]
Finally, let $p(d)$ defined as in Theorem \ref{thm:SMB}.
Then there exist positive constants $C,u_0$ such that, for all $n\in\N$ and for all $u\geq u_0$,
\begin{align*}
\MoveEqLeft (\gmu_\Psi \!\otimes\! \gmu_\Phi)\!\left\{ (\eta,\omega)\in\Omega\times\Omega : \frac{\log W_n(\eta,\omega) }{(2n+1)^d} > h(\mu_\Psi)+h(\mu_\Psi|\mu_\Phi)+\frac{u}{(2n+1)^{p(d)}}\right\}\\
& \leq C\,  \exp\left(-c_{\Phi,\Psi}\, (2n+1)\, u^2\right).
\end{align*}
Moreover, we have
\begin{align*}
\MoveEqLeft (\gmu_\Psi\! \otimes\! \gmu_\Phi)\!\left\{ (\eta,\omega)\in\Omega\times\Omega : \frac{\log W_n(\eta,\omega) }{(2n+1)^d} < h(\mu_\Psi)+h(\mu_\Psi|\mu_\Phi)-\frac{u}{(2n+1)^{p(d)}}\right\}\\
& \leq 
C\, \max\left\{\exp\left(-c_{\Phi,\Psi}\, (2n+1)\, u^2\right), \exp\left(-(2n+1)^{d-p(d)}\, \frac{u}{2}\right) \right\} 
\end{align*}
for all $n\in\N$ and for all $u\geq u_0$.
\end{theorem}

Let us make a few comments on this result.
The constant $u_0$ is the same as in Theorem  \ref{thm:SMB}.
Notice the dissymmetry between the two bounds when $n$ is fixed: the second bound then becomes exponentially small in $u$, not in $u^2$ as in the first bound. The second bound is of course useful only if
$\frac{u}{(2n+1)^{p(d)}}<h(\mu_\Psi)+h(\mu_\Psi|\mu_\Phi)$. Given $u\geq u_0$, this is always the case if $n$ is large enough.

\begin{proof}
We treat the case $\Phi=\Psi$. The other case follows the same lines of proof using Theorem \ref{thm:rel-ent} instead Theorem 
\ref{thm:SMB}.
The idea is to write
\[
\log W_n(\eta,\omega)=\log \big[W_n(\eta,\omega)\mu_\Phi(\EuScript{C}_n(\eta))\big]
-\log \mu_\Phi(\EuScript{C}_n(\eta)).
\]
Then we have the following obvious inequality.
\begin{align}
\label{est-OW}
\MoveEqLeft (\gmu_\Phi \otimes \gmu_\Phi)\left\{ (\eta,\omega)\in\Omega\times\Omega : 
\frac{\log W_n(\eta,\omega)}{(2n+1)^d} > h(\mu_\Phi)+u\right\}\\
\nonumber
& \leq (\gmu_\Phi \otimes \gmu_\Phi)\!\left\{ (\eta,\omega)\in\Omega\times\Omega : \frac{\log \big[W_n(\eta,\omega)\mu_\Phi(\EuScript{C}_n(\eta))\big]}{(2n+1)^d} > \frac{u}{2}\right\}\\
& \quad \; +
\gmu_\Phi\left\{ \eta\in\Omega : -\frac{\log \mu_\Phi(\EuScript{C}_n(\eta))}{(2n+1)^d} > h(\mu_\Phi)+\frac{u}{2}\right\}.
\nonumber 
\end{align}
We now control each term separately. To control the first one, we use Theorem 2.2. in \cite{ACRV} which we formulate here with our notations and under a form suitable for our purposes. Let $a_n$ be any pattern supported on $C_n$. Define $T_{a_n}(\omega)$ as the volume of the smallest hypercube $C_k$ which contains the support of a pattern congruent to $a_n$.
Then there exist positive constants $c_1,c_2,\lambda_{a_n},\lambda_1,\lambda_2$ such that
$\lambda_{a_n}\in [\lambda_1,\lambda_2]$ and such that, for any $t>0$, one has
\begin{equation}\label{explaw1}
\gmu_\Phi\left\{\omega\in\Omega: T_{a_n}(\omega)>\frac{t}{\lambda_{a_n}\mu_\Phi([a_n])} \right\}
\leq c_1\, e^{-c_2 t}.
\end{equation}
By $[a_n]$ we mean the cylinder set made of all configurations $\xi$ such that $\xi_{C_n}=a_n$.
The first term in the r.h.s. of \eqref{est-OW} is equal to
\begin{align}
\nonumber
\MoveEqLeft \sum_{a_n\in \Omega_n} \gmu_{\Phi}([a_n])\, 
\gmu_{\Phi}\left\{\omega\in\Omega : T_{a_n}(\omega)\mu_\Phi([a_n]) > e^{(2n+1)^d\frac{u}{2}} \right\}\\
\label{coq}
& \leq
c_1\, \exp\big(-c_2\lambda_1 e^{(2n+1)^d\frac{u}{2}}\big).
\end{align}
where the inequality follows by \eqref{explaw1}. The second term in the r.h.s. of \eqref{est-OW} is estimated using
Theorem \ref{thm:SMB} from which it follows easily that this term is bounded above by 
\begin{equation}\label{eq:soupe}
 \exp\left(-\frac{ (1-\mathfrak{c}(\bgamma^{\Phi}))^2}{128\opnorm{\Phi}^2}\, (2n+1)^{1+2p(d)}\, u^2\right)
\end{equation}
for all $n\in\N$ and for all $u\geq u_0$, where $p(d)=1/2$ if $d=2$ and $p(d)=1$ if $d\geq 3$. The bound 
\eqref{eq:soupe} is much bigger than the bound \eqref{coq}, hence the first inequality of the theorem follows after rescaling $u$.\\
We now prove the other inequality of the theorem. We now have
\begin{align}
\label{est-OW-bis}
& (\gmu_\Phi \otimes \gmu_\Phi)\left\{ (\eta,\omega)\in\Omega\times\Omega : 
\frac{1}{(2n+1)^d}\log W_n(\eta,\omega) < h(\mu_\Phi)-u\right\}\\
\nonumber
& \leq (\gmu_\Phi \otimes \gmu_\Phi)\!\left\{ (\eta,\omega)\in\Omega\times\Omega : \frac{1}{(2n+1)^d}\log \big[W_n(\eta,\omega)\mu_\Phi(\EuScript{C}_n(\eta))\big] < \frac{u}{2}\right\}\\
& \quad \; +
\gmu_\Phi\left\{ \eta\in\Omega: -\frac{1}{(2n+1)^d}\log \mu_\Phi(\EuScript{C}_n(\eta)) < h(\mu_\Phi)-\frac{u}{2}\right\}.
\nonumber 
\end{align}
The second term in the r.h.s. is also bounded by \eqref{eq:soupe}. To get an upper bound for the first term in the
r.h.s., we need to use the following result proved in \cite[Lemma 4.3]{ACRV}: 
\[
\lambda_1 \leq -\frac{\log \gmu_\Phi\left\{\omega\in\Omega : T_{a_n}(\omega)>t\right\}}{t\mu_\Phi([a_n])}\leq
\lambda_2
\]
provided that $t\mu_\Phi([a_n])\leq \frac12$, and where $\lambda_1,\lambda_2$ are defined as above in this proof.
We get the upper bound
\[
\lambda_2\, \exp\left(-(2n+1)^d\, \frac{u}{2}\right).
\]
This ends the proof.
\end{proof}

Combining the results in \cite{CR} and Theorem \ref{thm-seb}, one could get the analog of Theorem \ref{thm:patterns} for the low-temperature Ising ferromagnet. But an extra work is needed to make some
of the constants involved in the estimates in \cite{CR} more explicit and we will not do this.

\section{Application 5: Bounding $\db$-distance by relative entropy}\label{sec:dbar}

Given $n\in\N$, define the (non normalized)
Hamming distance between $\omega$ and $\eta$ that belong to $\Omega_n$ by
\begin{equation}\label{def:Hamming}
\db_n(\omega,\eta)=\sum_{x\in C_n} \un_{\{\omega_x\neq\eta_x\}}.
\end{equation}
Given two shift-invariant probability measures $\mu,\nu$ on $\Omega$, denote by $\mu_n$ and $\nu_n$
their projections on $\Omega_n$. Next define the $\db$-distance between $\mu_n$ and $\nu_n$ by
\[
\db_n(\mu_n,\nu_n)=\inf_{\pee\!_{n}\in\, \mathcal{C}(\mu_n,\nu_n)}
\int_{\Omega_n}\int_{\Omega_n} \db_n(\omega,\eta)\, \dd\pee_n(\omega,\eta) 
\]
where $\mathcal{C}(\mu_n,\nu_n)$ denotes the set of all shift-invariant couplings of $\mu_n$ and $\nu_n$, that is,
the set of jointly shift-invariant probability measures on $\Omega_n\times \Omega_n$ with marginals $\mu_n$ and $\nu_n$.
One can prove (see {\em e.g.} \cite{RS}) that $\db_n(\mu_n,\nu_n)$ normalized by $(2n+1)^d$ converges to a limit that we denote
by $\db(\mu,\nu)$:
\[
\db(\mu,\nu)=\lim_{n\to\infty} \frac{\db_n(\mu_n,\nu_n)}{(2n+1)^d}.
\]
This defines a distance on the set of shift-invariant probability measures on $\Omega$.
We have the following result.
\begin{theorem}
\leavevmode\\ 
Let $\Phi\in\mathscr{B}_T$ and assume that the associated specification $\bgamma^{\Phi}$ satisfies Dobrushin's uniqueness condition \eqref{def:dobcoef}. 
Then, for every shift-invariant probability measure $\nu$
\begin{equation*}
\db(\mu_\Phi,\nu)\leq \frac{\sqrt{2}}{1-\mathfrak{c}(\bgamma^{\Phi})} \, \sqrt{h(\nu|\mu_\Phi)}
\end{equation*}
where $h(\nu|\mu_\Phi)$ is the relative entropy of $\nu$ with respect to $\mu_\Phi$
\textup{(}see \eqref{def:relent}\textup{)}.\\
Moreover, if $\nu=\mu_\Psi$ is also a Gibbs measure for a potential $\Psi\in\mathscr{B}_T$, then
\begin{equation}\label{eq:dbarre-bis}
\db(\mu_{\Phi},\mu_\Psi)\leq  \frac{2\sqrt{2}}{1-\mathfrak{c}(\bgamma^{\Phi})} \, \sqrt{\opnorm{\Phi-\Psi}}.
\end{equation}
\end{theorem}

Take for instance a finite-range potential $\Phi$ and $\beta_1,\beta_2$ such that $\beta_1<\beta_2$ 
with $\beta_2$ small enough to be in Dobrushin's uniqueness regime. Then the previous inequality reads
\[
\db(\mu_{\beta_1\Phi},\mu_{\beta_2\Phi})\leq 
\frac{2\sqrt{2}\sqrt{\opnorm{\Phi}}}{1-\mathfrak{c}(\bgamma^{\beta_1\Phi})} \, \sqrt{\beta_2-\beta_1}.
\]

Before proving the previous theorem, let us introduce a certain set of Lipschitz functions.
Given $n\in\N$ and let $F:\Omega\to\R$ be a cylindrical function with dependence set $C_n$. We have
\[
|F(\omega)-F(\omega')|\leq \sum_{x\in C_n} \un_{\{\omega_x\neq \omega'_x\}} \delta_x(F).
\]
Assume that $\delta_x(F)=1$ for all $x\in C_n$. In particular 
$
\|\ushort{\delta}(F)\|_2^2\leq (2n+1)^d.
$
We can identify this function with a $1$-Lipschitz function on $\Omega_n$
with respect to the distance \eqref{def:Hamming}. Denote by $\text{Lip}_{1,\mu_{\Phi,n}}(\Omega)$ the set of functions $F$ which are $1$-Lipschitz and such that $\E_{\mu_{\Phi,n}}[F]=0$. (Recall that $\mu_{\Phi,n}$
is the Gibbs measure associated to $\Phi$ induced on $\Omega_n$ by projection. 

\begin{proof}
We now use a general theorem (see \cite[p. 5]{bg} or \cite[p. 101]{blm}). In the present setting, it states that the property that there exists a constant $b>0$ such that
\begin{equation}
\label{plouf}
\E_{\mu_{\Phi,n}}[\exp(u F)]\leq \exp(bu^2),\quad \forall u\in\R,\;
\forall F\in  \text{Lip}_{1,\mu_{\Phi,n}}(\Omega)
\end{equation}
is equivalent to the property that, for all probability measures $\nu_n$ on $\Omega_n$, we have
\begin{equation}\label{eq:dbar-entrop}
\db_n(\mu_{\Phi,n},\nu_n)\leq 2\sqrt{b H_n(\nu_n|\mu_{\Phi,n})}.
\end{equation}
By Theorem \ref{thm:gemb-dob} we know that $\mu_{\Phi,n}$ satisfies \eqref{plouf} with
\[
b=\frac{\|\ushort{\delta}(F)\|_2^2}{2(1-\mathfrak{c}(\bgamma^{\Phi}))^2}
\leq \frac{(2n+1)^d}{2(1-\mathfrak{c}(\bgamma^{\Phi}))^2}
\]
Hence \eqref{eq:dbar-entrop} reads
\[
\db_n(\mu_{\Phi,n},\nu_n)\leq 2\sqrt{D (2n+1)^d H_n(\nu_n|\mu_{\Phi,n})}\, .
\]
Dividing both sides by $(2n+1)^d$ and taking the limit $n\to\infty$ gives the announced inequality.

To prove inequality \eqref{eq:dbarre-bis}, we use \eqref{formule-rel-entropy} and \eqref{VP} (applied to $\Psi$)
to get
\[
h(\mu_\Psi|\mu_\Phi)= P(\Phi)+\E_{\mu_\Psi}[f_\Phi]-h(\mu_\Psi)=
P(\Phi)-P(\Psi)+\E_{\mu_\Psi}[f_\Phi]-\E_{\mu_\Psi}[f_\Psi].
\]
The desired inequality follows from the following facts:
\[
|P(\Phi)-P(\Psi)|\leq \opnorm{\Phi-\Psi}
\]
and
\[
|\E_{\mu_\Psi}[f_\Phi]-\E_{\mu_\Psi}[f_\Psi]|\leq \| f_\Phi-f_\Psi\|_\infty\leq \opnorm{\Phi-\Psi}.
\]
The theorem is proved.
\end{proof}

\section{Application 6: Fattening patterns}\label{sec:H}

We can naturally generalize the Hamming distance defined in \eqref{def:Hamming}) 
as follows. Let $\Lambda\in \mathcal{P}$ (finite subset of $\Zd$) and define
\[
\db_\Lambda(\omega,\eta)=\sum_{x\in \Lambda} \un_{\{\omega_x\neq\eta_x\}}.
\]
Given a subset $\EuScript{B}_\Lambda\subset \Omega_\Lambda$ define
\[
\db_\Lambda(\omega,\EuScript{B}_\Lambda)=\inf_{\omega'\in \EuScript{B}_\Lambda}\db_n(\omega,\omega').
\]
Given $\epsilon>0$, define the ``$\epsilon$-fattening'' of $\EuScript{B}_\Lambda$ as
\[
\EuScript{B}_{\Lambda,\epsilon}=\big\{\omega\in \Omega_\Lambda : 
\db_\Lambda(\omega,\EuScript{B}_n)\leq \epsilon|\Lambda|\big\}.
\]
We have the following abstract result.

\begin{theorem}
\leavevmode\\ 
Let $\Lambda\in\mathcal{P}$. Suppose that $\nu$ is a probability measure which satisfies $\gcb{D}$ and such that $\nu(\EuScript{B}_\Lambda)=\frac{1}{2}$. Then, we have 
\begin{equation}\label{eq:concen-H}
\nu\big(\EuScript{B}_{\Lambda,\epsilon}\big)\geq
1-\exp\left[-\frac{|\Lambda|}{4D}\left(\epsilon-\frac{2\sqrt{D\ln 2}}{\sqrt{|\Lambda|}}\,\right)^2\, \right]
\end{equation}
whenever $\epsilon>\frac{2\sqrt{D \ln 2}}{\sqrt{|\Lambda|}}$.
\end{theorem}

We take $\nu(\EuScript{B}_\Lambda)=\frac{1}{2}$ for the sake of definiteness. One can take
$\nu(\EuScript{B}_\Lambda)=\alpha\in (0,1)$ and replace $\ln 2$ by $\ln \alpha^{-1}$ in \eqref{eq:concen-H}.
The previous theorem can be loosely phrased as follows: For a probability measure satisfying a Gaussian concentration bound, if we ``fatten'' a bit a set of patterns which represents, say, half of the mass of $\Omega_\Lambda$, what is left has an extremely small mass.

\medskip

\begin{proof}
Consider the local function $F(\omega)=\db_\Lambda(\omega,\EuScript{B}_\Lambda)$. One easily checks that
$\delta_x(F)\leq 1$ for all $x\in \Lambda$. Applying \eqref{eq:expcon1}  gives
\begin{equation}\label{puic}
\nu\left\{\omega\in \Omega : F(\omega) \geq u+\E_\nu [F]\right\} \leq
\exp\left(-\frac{u^2}{4D|\Lambda|}\right)
\end{equation}
for all $u>0$. We now estimate $\E_\nu [F]$. Applying \eqref{eq:gemb} to $-\lambda F$ ($u\in\R$) we get
\[
\exp\left(\lambda\E_\nu[F]\right) \E_\nu\big[\exp\left(-\lambda F\right) \big]
\leq \exp\left(D\lambda^2 |\Lambda|\right).
\]
Observe that by definition of $F$ we have
\[
\E_\nu\big[\exp\left(-\lambda F\right) \big]\geq 
\E_\nu\big[\un_{\EuScript{B}_\Lambda}\exp\left(-\lambda F\right) \big]=\nu(\EuScript{B}_\Lambda). 
\]
Combining these two inequalities and taking the logarithm gives
\[
\E_\nu[F]\leq \inf_{\lambda>0} \Big\{D \lambda|\Lambda| +
\frac{1}{\lambda}\ln \big(\nu(\EuScript{B}_\Lambda)^{-1}\big)\Big\},
\]
{\em i.e.},
\[
\E_\nu[F]\leq 2\sqrt{D|\Lambda| \ln\big(\nu(\EuScript{B}_\Lambda)^{-1}\big)}:=E.
\]
Therefore inequality \eqref{puic} implies that
\[
\nu\left\{\omega\in \Omega : F(\omega) \geq u'\right\} \leq
\exp\left(-\frac{(u'-E)^2}{4D|\Lambda|}\right)
\]
for all $u'>E$. To finish the proof, take $u'=\epsilon|\Lambda|$ and observe that
$\nu\left\{\omega\in \Omega : F(\omega) \geq u'\right\}=\nu\big(\EuScript{B}_{\Lambda,\epsilon}^c\big)$.
\end{proof}

\begin{corollary}
\leavevmode\\ 
Let $\Phi\in\mathscr{B}_T$ and assume that the associated specification $\bgamma^{\Phi}$ satisfies Dobrushin uniqueness condition \eqref{def:dobcoef}. Then \eqref{eq:concen-H} holds with $D=\frac{1}{2(1-\mathfrak{c}(\bgamma^{\Phi}))^2}$.
\end{corollary}

\begin{remark}
Inequality \eqref{eq:concen-H} can also be deduced from \eqref{eq:dbar-entrop} by an argument due to Marton \cite{marton0}. But this kind of argument does not  work when one has only moment inequalities because there is no analog of \eqref{eq:dbar-entrop}, to the best of our knowledge.
\end{remark}

We now turn to the situation when one has moment inequalities.

\begin{theorem}
\leavevmode\\ 
Let $\Lambda\in\mathcal{P}$.
For $\nu$ satisfying $\mcb{2p,C_{2p}}$ and such that $\nu(\EuScript{B}_\Lambda)=\frac{1}{2}$, we have
\[
\nu\big(\EuScript{B}_{\Lambda,\epsilon}\big)\geq
1-\frac{C_{2p}}{|\Lambda|^{p/d}} \left(\epsilon-\frac{(2C_{2p})^\frac{1}{2p}}{\sqrt{|\Lambda|}} \right)^{-2p}
\]
whenever $\epsilon>0$ and $n\in\N$ are such that $\epsilon>\frac{(2C_{2p})^\frac{1}{2p}}{\sqrt{|\Lambda|}}$.
\end{theorem}

\begin{proof}
As in the previous proof, consider the local function $F(\omega)=\db_\Lambda(\omega,\EuScript{B}_\Lambda)$
which is such that $\delta_x(F)\leq 1$ for all $x\in \Lambda$. Applying \eqref{eq:momentdev} we get
\[
\nu\left\{\omega\in \Omega : F(\omega) \geq u+\E_\nu [F]\right\} \leq \frac{C_{2p}|\Lambda|^{p}}{u^{2p}}
\]
for all $u>0$.
We easily obtain an upper bound for $\E_\nu [F]$ by using \eqref{eq:momentb} and the fact that $F\equiv 0$ on $\EuScript{B}_\Lambda$ :
\[
\nu(\EuScript{B}_\Lambda) (\E_\nu [F])^{2p}=
\E_\nu\left[ \un_{\EuScript{B}_\Lambda}(F-\E_\nu [F])^{2p}\right]\leq C_{2p}|\Lambda|^{p}\,.
\]
whence
\[
\E_\nu [F]\leq \frac{C_{2p}^{\frac{1}{2p}}\sqrt{|\Lambda|} }{\nu(\EuScript{B}_\Lambda)^{\frac{1}{2p}}}.
\]
We finish in the same way as in the previous proof to get the desired inequality.
\end{proof}

In view of Theorem \ref{thm:allmomentsIsing}, the previous theorem applies to the plus-phase of the Ising model at sufficiently low temperature. Moreover, we can optimize over $p$. In fact, applying the stretched-exponential concentration inequality that holds in this case, we have indeed a stronger result.

\begin{theorem}
\leavevmode\\ 
Let $\gmu_\beta^+$ be the plus-phase of the low-temperature Ising model. Take $\Lambda\in \mathcal{P}$ such that
$\gmu_\beta^+(\EuScript{B}_\Lambda)=\frac12$. Then there exists $\bar{\beta}>\beta_c$ such that, for each $\beta>\bar{\beta}$, there exists $\varrho=\varrho(\beta)\in (0,1)$ and two positive constants $c_\varrho$ and $c'_\varrho$ such that 
\[
\gmu_\beta^+\big(\EuScript{B}_{\Lambda,\epsilon}\big)\geq
1- 4 \exp\left[ -c_\varrho\, |\Lambda|^{\frac{\varrho}{2}}
\left(\epsilon -\frac{c'_\varrho}{\sqrt{|\Lambda|}}\right)^\varrho
\, \right]
\]
whenever $\epsilon>0$ and $n\in\N$ are such that $\epsilon >\frac{c'_\varrho}{\sqrt{|\Lambda|}}$.
\end{theorem}

\begin{proof}
Consider the local function $F(\omega)=\db_\Lambda(\omega,\EuScript{B}_\Lambda)$ which is such that
$\delta_x(F)\leq 1$ for all $x\in \Lambda$. We apply Theorem \ref{thm-seb}.
Using \eqref{eq:devineqsubexp} we get
\[
\gmu_\beta^+\left\{\omega\in \Omega :  F(\omega)\geq u+\E_{\mu_\beta^+} [F]\right\} \leq
4\, \exp\left(-\frac{c_\varrho u^{\varrho}}{|\Lambda|^\frac{\varrho}{2}}\right),
\]
for all $u>0$. We now estimate $\E_{\mu_\beta^+} [F]$ from above by using \eqref{LpOrlicz},
\eqref{eq:lux-ising} and the fact that $F\equiv 0$ on $\EuScript{B}_\Lambda$ :
\begin{equation}\label{eq:fifu}
K\!_\varrho\, \sqrt{|\Lambda|}\geq \|F-\E_{\mu_\beta^+}[F]\|_{M\!_\varrho}\geq
B\!_\varrho^{-1}\E_{\mu_\beta^+} [F]
\sup_{q\in 2\N} (2^{-1/q}q^{-1/\varrho}).
\end{equation}
The function $\theta : \R^+\backslash\{0\} \to \R^+$ defined by $\theta(u)=2^{-1/u}u^{-1/\varrho}$ has a unique maximum at $u=\varrho \ln 2<2$. Hence we take $q=2$ in the right-hand side of \eqref{eq:fifu}, which gives
\[
\E_{\mu_\beta^+} [F]\leq 2^{\frac{1}{\varrho}+\frac12}K\!_\varrho \, B\!_\varrho \, (2n+1)^\frac{d}{2}.
\]
The rest of the proof is the same as in the previous proofs and we obtain the announced inequality with
$c'_\varrho=2^{\frac{1}{\varrho}+\frac12}K\!_\varrho \, B\!_\varrho$. 
\end{proof}

\section{Application 7: Almost-sure central limit theorem}\label{sec:ASCLT}

In this section we show how to use concentration inequalities to get a limit theorem.
We consider limits along cubes but we can generalize without further effort to van Hove sequences.

\subsection{Some preliminary definitions}

\begin{definition}\label{def:clt}
\leavevmode\\
Let $\nu$ be a shift-invariant probability measure and let $f:\Omega\to\R$ be a continuous function
such that $\int f\dd\nu=0$.
We say that $(f,\nu)$ satisfies the central limit theorem with variance $\sigma_f^2$
if there exists a number $\sigma_f\geq 0$ such that for all $u\in\R$
\[
\lim_{n\to\infty}\nu\left\{
\omega\in \Omega :
\frac{\sum_{x\in C_n}f(T_x \omega)}{(2n+1)^{\frac{d}{2}}} \leq u\right\}
=
\frac{1}{\sigma\!_f\sqrt{2\pi}}\int_{-\infty}^u \exp\left(-\frac{v^2}{2\sigma\!_f^2}\right) \dd v.
\]
\end{definition}

As a convention, we define the right-hand side to be the Dirac mass at $0$ if $\sigma_f=0$.
There is of course no loss of generality in considering continuous functions such that
$\int f\dd\nu=0$. In the cases we are going to consider, one has
\begin{equation}\label{eq:sigma}
\sigma_f^2=\sum_{x\in\Zd} \int f \cdot f\circ T_x \dd\nu<\infty.
\end{equation}

We need the following convenient definition.
\begin{definition}[Summable decay of correlations]
\leavevmode\\
Given a shift-invariant probability measure $\nu$ and a continuous function $f$ such that $\int f\dd \nu=0$, we say that we have a summable decay of correlations if
\begin{equation}
\label{def:summable-cov}
\mathlarger{\sum_{x\in\Zd}} \int |f \cdot f\circ T_x| \dd\nu<+\infty.
\end{equation}
\end{definition}
It follows from \eqref{def:summable-cov} that for all $n\in\N$
\begin{equation}\label{eq:est-variance}
\int \Big( \sum_{x\in C_n} f\circ T_x\Big)^2 \dd\nu\leq C\, (2n+1)^{d}
\end{equation}
where $C=\mathlarger{\sum_{x\in\Zd}} \int |f \cdot f\circ T_x| \dd\nu$.

The almost-sure central limit theorem is about replacing the convergence in law by the almost-sure convergence (in weak topology) of the following empirical logarithmic average:
\begin{equation}\label{def:log-av}
\frac{1}{L_N}\, \mathlarger{\sum_{n=1}^N}\, \frac{1}{n}\, \ggdelta_{\frac{\sum_{x\in C_n}f(T_x \omega)}{(2n+1)^{d/2}}}
\end{equation}
where
\[
L_N=\sum_{n=1}^N\frac{1}{n}=\ln N + \Oun.
\]

For each $N\in\N$ and $\omega\in \Omega$, \eqref{def:log-av} defines a probability measure on $\R$. 
Our goal is to prove that it converges, for $\nu$-almost every $\omega$, to the Gaussian measure $G_{0,\sigma_f^2}$ defined by
\[
\dd G_{0,\sigma\!_f^2}(v)=\frac{1}{\sigma_f\sqrt{2\pi}} \, \exp\Big(-\frac{v^2}{2\sigma\!_f^2}\Big)\dd v, \; v\in\R.
\] 
When such a convergence takes place, one says that $(f,\nu)$ satisfies the almost-sure central limit theorem.
We shall prove a stronger result: the convergence will be with respect to the Kantorovich distance
$\distk$ which is defined as follows.
Let
\[
\mathscr{L}=\{\rho:\R\to\R : \rho\;\textup{is {\em 1}-Lipschitz}\}.
\]
For two probability measures $\lambda,\lambda'$ on $\R$, let
\[
\distk(\lambda,\lambda')
=\sup_{\rho\in \mathscr{L}_0} \left(\int \rho \dd\lambda-\int \rho \dd\lambda'\right)
\]
where $\mathscr{L}_0$ is the set of functions in $\mathscr{L}$ vanishing at the origin.
We can replace $\mathscr{L}$ by $\mathscr{L}_0$ in the definition of the distance because we consider probability measures.
This distance metrizes the weak topology on the set of probability measures  on $\R$ such that $\int_{\R} d(u_0,u) \dd\lambda(u)<\infty$ (where $u_0\in\R$ is an arbitrary chosen point).

\subsection{An abstract theorem and some applications}

The following abstract theorem says that, if the central limit theorem holds and if we have $\mcb{2,C_2}$, then we have an almost-sure central limit theorem. In fact, the convergence to the Gaussian measure is with respect to the Kantorovich distance, which is stronger than weak convergence.

\begin{theorem}\label{thm:asclt}
\leavevmode\\
Let $f\in\Delta_1(\Omega)$ and $\nu$ be a shift-invariant probability measure.
Assume that the following conditions hold:
\begin{enumerate}
\item
$(f,\nu)$ satisfies the central limit theorem with variance $\sigma_f^2>0$
(in the sense of Definition \ref{def:clt}); 
\item
$\nu$ satisfies $\mcb{2,C_2}$ (in the sense of Definition \ref{def:moments});
\item
the decay of correlations is summable in the sense of \eqref{def:summable-cov}.
\end{enumerate}
Then, for $\nu$-almost every $\omega\in \Omega$, 
\[
\lim_{N\to\infty}\distk\big(\mathcal{A}_{N,\,\omega},G_{0,\sigma_f^2}\big)=0.
\]
\end{theorem}

We now apply this theorem in two situations, namely under Dobrushin's uniqueness condition, and for the
low-temperature Ising ferromagnet.

\begin{theorem}\label{thm:dob-asclt}
\leavevmode\\
Let $\Phi\in\mathscr{B}_T$ and assume that the associated specification $\bgamma^{\Phi}$ satisfies  Dobrushin's uniqueness condition \eqref{def:dobcoef}.
Moreover, assume that
\begin{equation}\label{dobddelta}
\sum_{x\in\Zd} \|x\|_\infty^{d+\delta}\, C_{0,x}(\bgamma^\Phi) <+\infty
\end{equation}
for some $\delta>0$, and that $f\in C^0(\Omega)$ satisfies
\begin{equation}\label{deltafd}
\sum_{x\in\Zd} \|x\|_\infty^d\, \delta_x(f) <+\infty.
\end{equation}
Without loss of generality, assume that $\int f \dd\mu_\Phi=0$.
Then, for $\mu_\Phi$-almost every $\omega\in \Omega$, 
\[
\lim_{N\to\infty}\distk
\left(\frac{1}{L_N}\, \mathlarger{\sum_{n=1}^N}\, \frac{1}{n}\,
\ggdelta_{\frac{\sum_{x\in C_n} f(T_x\omega)}{(2n+1)^{d/2}}}\, ,G_{0,\sigma_f^2}\right)=0
\]
where $\sigma_f^2\in\left[0,\infty\right[$ is given by \eqref{eq:sigma}.
\end{theorem}

\begin{proof}
The conditions \eqref{dobddelta} and \eqref{deltafd} imply \eqref{def:summable-cov}.
The theorem is a direct consequence of Theorem \ref{thm:asclt} and Theorem 4.1 in \cite{kunsch}.
\end{proof}

The assumptions of the previous theorem are for instance satisfied if $\Phi$ is a finite-range potential with $\beta$ small enough and for any local function $f$. Let us state a corollary for the empirical magnetization $M_n(\omega)=\sum_{x\in C_n}s_0(T_x \omega)$, where $s_0(\omega)=\omega_0$, in the case of spin pair potentials \eqref{def-sum-pair-pot}.

\begin{corollary}
Consider a ferromagnetic spin pair potential $\beta\Phi$ such that $\sum_{x\in\Z^d} \tanh(\beta J(x))<1$. Assume that 
\begin{equation}\label{eq:jj}
\beta \sum_{x\in\Zd} \|x\|_\infty^{d+\delta}\, J(x) <+\infty
\end{equation}
for some $\delta>0$. Then, for $\gmu_{\beta\Phi}$-almost every $\omega\in\Omega$, we have
\[
\lim_{N\to\infty}\distk\left(\frac{1}{\ln N}\, \mathlarger{\sum_{n=1}^N}\, \frac{1}{n}\, \ggdelta_{M_n(\omega)/(2n+1)^{\frac{d}{2}}}, G_{0,\sigma_\beta^2}\right)=0
\]
where 
\[
\sigma_\beta^2=\mathlarger{\sum_{x\in\Zd}} \int s_0\cdot s_0\circ T_x  \dd \gmu_{\beta\Phi} \in \left]0,\infty\right[.
\]
\end{corollary}

Recall that in the regime considered in this corollary we have $\E_{\mu_{\beta\Phi}}[s_0]=0$.
Observe that, for sufficiently high temperature, condition \eqref{eq:jj} implies $\sum_{x\in\Z^d} \tanh(\beta J(x))<1$. It is well known (see \cite{gross}) that in Dobrushin's uniqueness regime one has for each $\beta$
\[
\left| \int s_0\cdot s_0\circ T_x\, \dd \gmu_{\beta\Phi} \right| \leq C\,  \|x\|_\infty^{-(d+\delta)}, 
\]
where $C>0$ is independent of $x$. (Recall that $\int s_0 \dd \gmu_{\beta\Phi}=0$ for $\beta<\beta_c$.)

The next theorem is an almost-sure central limit theorem  for the empirical magnetization in the low-temperature Ising ferromagnet.

\begin{theorem}\label{thm:ising-low-asclt}
\leavevmode\\
Let $\gmu_\beta^+$ be the plus phase of the low-temperature Ising model. 
Then there exists  $\bar{\beta}$ such that, for each $\beta>\bar{\beta}$ and for $\gmu_\beta^+$-almost every $\omega\in\Omega$, we have
\[
\lim_{N\to\infty}\distk\left(\frac{1}{\ln N}\, \mathlarger{\sum_{n=1}^N}\, \frac{1}{n}\, \ggdelta_{(M_n(\omega)-\E_{\mu_\beta^+}[s_0])/(2n+1)^{\frac{d}{2}}}, G_{0,\sigma\!_\beta^2}\right)=0
\]
where 
\[
\sigma\!_\beta^2=\mathlarger{\sum_{x\in\Zd}} \int s_0\cdot s_0\circ T_x \dd \gmu_\beta^+ \in \left]0,\infty\right[.
\]
\end{theorem}

\begin{proof}
The theorem follows at once from Theorem \ref{thm:allmomentsIsing}, \cite{martinlof} and Theorem \ref{thm:asclt}.
\end{proof}

\subsection{Proof of the abstract theorem}

We now prove Theorem \ref{thm:asclt}. Throughout the proof, we use the notations
\[
S_nf=\sum_{x\in C_n}f\circ T_x
\quad
\text{and}
\quad
\mathcal{A}_{N,\omega}=
\frac{1}{L_N}\,\mathlarger{\sum_{n=1}^N}\, \frac{1}{n}\, \ggdelta_{\frac{S_nf(\omega)}{(2n+1)^{d/2}}}\, .
\]
{\em First step}. We are going to prove that
\begin{equation}\label{eq:esp-tend-0}
\lim_{N\to\infty} \E_\nu\big[\distk\big(\mathcal{A}_{N,\cdot},G_{0,\sigma_f^2}\big)\big]=0.
\end{equation}
Let $B>0$. Since for any $\rho\in\mathscr{L}_0$ one has $|\rho(v)|\leq |v|$ for all $v$ we have
\begin{align}
\nonumber
\MoveEqLeft[6]
\distk\big(\mathcal{A}_{N,\omega},G_{0,\sigma_f^2}\big)  \leq 
\sup_{\rho\, \in \mathscr{L}_0}  \int_{-B}^B \rho(v) \big(\dd\mathcal{A}_{N,\omega}(v)-\dd G_{0,\sigma_f^2}(v)\big)
\\
\label{eq:A}
& \quad\quad
+ \int_{\{|v|>B\}} |v|\dd\mathcal{A}_{N,\omega}(v)
+\int_{\{|v|>B\}} |v|\dd G_{0,\sigma_f^2}(v).
\end{align}
The last integral is obviously bounded by $c_1/B$ where $c_1>0$ depends only on $f$. (It is indeed much smaller but this bound suffices.)
We now bound the expectation of the second term in the r.h.s., uniformly in $N$. 
Using \eqref{eq:est-variance} and the inequality
\[
\E_\nu\left[ \un_{(B,+\infty)}(Z)\, Z\right]
\leq \frac{\E_\nu[Z^2]}{B},
\]
which follows from Cauchy-Schwarz inequality and Bienaym\'e-Chebyshev inequality, we get
\begin{align}
\nonumber
\MoveEqLeft \E_\nu\Big[ \int_{\{|v|>B\}} |v|\dd\mathcal{A}_{N,\cdot}(v)\Big]\\
\nonumber
&=\frac{1}{L_N}\sum_{n=1}^N \frac{1}{n}\,
\E_\nu\left[ \un_{(B,+\infty)}\left(\frac{|S_nf|}{(2n+1)^{d/2}}\right)\frac{|S_nf|}{(2n+1)^{d/2}}\right]\\
\label{eq:1}
& \leq \frac{c_2}{B}
\end{align}
where $c_2>0$ is independent of $N$ and $B$.
We turn to the first term in the r.h.s. of \eqref{eq:A}. Since $[-B,B]$ is compact, we can apply Arzel\`a-Ascoli theorem
to conclude that $\mathscr{L}_0$ is precompact in the uniform topology. As a consequence, given $\epsilon>0$, there exists a positive integer $r=r(\epsilon)$ and functions $\tilde\rho_j:[-B,B]\to\R$
in $\mathscr{L}_0$, $j=1,\ldots,r$, such that,  for any $\rho\in \mathscr{L}_0$, there is at least one integer $1\leq j\leq r$ such that
\[
\sup_{|v|\leq B}|\rho(v)-\tilde\rho_j(v)|\leq \epsilon.
\]
Therefore we have
\begin{align}
\nonumber
& \sup_{\rho\in \mathscr{L}_0}  \int_{-B}^B \rho(v) \big(\dd\mathcal{A}_{N,\omega}(v)-\dd G_{0,\sigma_f^2}(v)\big)\\
\label{eq:2}
& \qquad\qquad \leq \sup_{1\leq j\leq r} \int_{-B}^B \tilde\rho_j(v) \big(\dd\mathcal{A}_{N,\omega}(v)-\dd G_{0,\sigma_f^2}(v)\big)+2\epsilon.
\end{align}
To proceed, we need to define, for each function $\tilde\rho_j$, a function $\rho_j\in \mathscr{L}_0$ defined on $\R$ and coinciding with $\tilde\rho_j$ on $[-B,B]$. This is done by setting
\[
\rho_j(v)=
\begin{cases}
0 & \text{if} \; v<-B-|\tilde\rho_j(-B)| \\
\tilde\rho_j(-B)+\sign(\tilde\rho_j(-B))(v+B)  & \text{if}\; -B-|\tilde\rho_j(-B)|\leq v<-B\\
\tilde\rho_j(v) & \text{if} \; v\in[-B,B]\\
\tilde\rho_j(B)-\sign(\tilde\rho_j(B))(v-B) & \text{if}\; B<v\leq B+|\tilde\rho_j(B)|\\
0 & \text{if} \quad v>B+|\tilde\rho_j(B)|.
\end{cases}
\]
Next, for each $1\leq j\leq r$ and for each $N\geq 1$, introduce the functions
\[
\widetilde F_N^{(j)}(\omega)
=\int_{-B}^B \tilde\rho_j(v) \big(\dd\mathcal{A}_{N,\omega}(v)-\dd G_{0,\sigma_f^2}(v)\big)
\]
and
\begin{align*}
F_N^{(j)}(\omega)
& =\int \rho_j(v) \big(\dd\mathcal{A}_{N,\omega}(v)-\dd G_{0,\sigma_f^2}(v)\big)\\
& = \frac{1}{L_N}\sum_{n=1}^N \frac{1}{n}\left[
\rho_j\left(\frac{S_nf(\omega)}{(2n+1)^{d/2}} \right) -\int \rho_j \dd G_{0,\sigma_f^2}
\right].
\end{align*}
We have
\begin{equation}\label{eq:FFR}
\widetilde F_N^{(j)}(\omega)=F_N^{(j)}(\omega)+R_N^{(j)}(\omega)
\end{equation}
where
\begin{equation}\label{eq:EFFR}
\E_\nu\left(\sup_{1\leq j\leq r}|R_N^{(j)}(\omega)|\right) \leq \frac{c_3}{B}
\end{equation}
where $c_3>0$ is independent of $N$, $r$ and $B$.
This estimate is proved as above (see \eqref{eq:A} and \eqref{eq:1}).
We now estimate the variance of $F_N^{(j)}$. Given $z\in\Zd$ we have
\begin{equation}\label{def-vzn}
\delta_z(F_N^{(j)})\leq \boldsymbol{v}_z^{(N)}:=
\frac{1}{L_N} \,\mathlarger{\sum_{n=1}^N}\, \frac{1}{n} \frac{1}{(2n+1)^{d/2}}\, 
\sum_{x\in C_n}\delta_{z-x}(f).
\end{equation}
The r.h.s. term can be rewritten as
$
\sum_{x\in \Zd}\, \delta_{z-x}(f)\,\boldsymbol{u}^{(N)}_x
$
where
\[
\boldsymbol{u}^{(N)}_x:=\mathlarger{\sum_{n=1}^N}\, \frac{1}{n(2n+1)^{d/2}L_N}\, \un_{C_n}(x)\un_{C_N(x)}.
\]
We now apply Young's inequality to get
\[
\|\ushort{\delta}(F_N^{(j)})\|_2^2 \leq \|\ushort{\delta}(f)\|_1^2\, \|\boldsymbol{u}^{(N)}\|_2^2.
\]
Since $f\in \Delta_1(\Omega)$ by assumption, $ \|\ushort{\delta}(f)\|_1^2<+\infty$, and it remains to
estimate $\|\ushort{\boldsymbol{u}}^{(N)}\|_2^2$. We have
\begin{align*}
\|\ushort{\boldsymbol{u}}^{(N)}\|_2^2
&= \frac{1}{L_N^2} \sum_{x\in C_N} \sum_{n=1}^N \sum_{m=1}^N \frac{1}{n}
\frac{1}{m} \frac{1}{(2n+1)^{d/2}} \frac{1}{(2m+1)^{d/2}}\ \un_{C_n}(x)\un_{C_m}(x)\\
&= \frac{2}{L_N^2} \sum_{x\in C_N} \sum_{n=1}^N \sum_{m=n}^N \frac{1}{n}
\frac{1}{m} \frac{1}{(2n+1)^{d/2}} \frac{1}{(2m+1)^{d/2}}\ \un_{C_n}(x)\un_{C_m}(x)\\
&= \frac{2}{L_N^2} \sum_{x\in C_N} \sum_{n=1}^N \sum_{m=n}^N \frac{1}{n}
\frac{1}{m} \frac{1}{(2n+1)^{d/2}} \frac{1}{(2m+1)^{d/2}}\ \un_{C_n}(x)\\
&= \frac{2}{L_N^2}  \sum_{n=1}^N \sum_{m=n}^N \frac{1}{n}
\frac{1}{m} \frac{1}{(2n+1)^{d/2}} \frac{1}{(2m+1)^{d/2}}\ \sum_{x\in C_N}\un_{C_n}(x)\\
&= \frac{2}{L_N^2}  \sum_{n=1}^N \sum_{m=n}^N \frac{1}{n}
\frac{1}{m} \frac{(2n+1)^{d/2}}{(2m+1)^{d/2}}\\
& \leq \frac{c}{L_N^2} \, \sum_{n=1}^N n^{\frac{d}{2}-1}\, \sum_{m=n}^N \frac{1}{m^{\frac{d}{2}+1}}\\
& \leq \frac{c}{L_N^2} \, \sum_{n=1}^N \frac{1}{n}=\frac{c}{L_N}\, ,
\end{align*}
where $c>0$ does not depend on $N$.
Hence, for any $1\leq j\leq r$, we have
\[
\|\ushort{\delta}(F_N^{(j)})\|_2^2
\leq \frac{c\, \|\ushort{\delta}(f)\|_1^2}{L_N}.
\]
Since we assumed that $\nu$ satisfies $\mcb{2,C_2}$, we end up with the
following estimate for the variance of $F_N^{(j)}$:
\begin{equation}\label{eq:var-Fj}
\E_\nu\left[ \left(F_N^{(j)}-\E_\nu[F_N^{(j)}]\right)^2\right]\leq  \frac{c\, C_2 \|\ushort{\delta}(f)\|_1^2}{L_N}.
\end{equation}
We now use \eqref{eq:FFR}, \eqref{eq:EFFR}, Cauchy-Schwarz inequality and
\eqref{eq:var-Fj} to obtain
\begin{align*}
\MoveEqLeft \E_\nu\left[\sup_{1\leq j\leq r} \widetilde{F}_N^{(j)}\right]\\
&\leq \E_\nu\left[\sup_{1\leq j\leq r} F_N^{(j)}\right] +\E_\nu\left[\sup_{1\leq j\leq r} R_N^{(j)}\right]\\
&\leq \E_\nu\left[\sum_{j=1}^r \big|F_N^{(j)}\big|\right]+ \frac{c_3}{B}\\
& \leq \sum_{j=1}^r\E_\nu\left[ \big|F_N^{(j)}-\E_\nu\big[F_N^{(j)}\big]\big|\right]
+ \sum_{j=1}^r \left|\E_\nu\left[F_N^{(j)}\right]\right| + \frac{c_3}{B}\\
& \leq \sum_{j=1}^r \left( \E_\nu\left[ \left(F_N^{(j)}-\E_\nu\big[F_N^{(j)}\big]\right)^2\right]\right)^{\frac{1}{2}}
+ \sum_{j=1}^r \big|\E_\nu\big[F_N^{(j)}\big]\big|+ \frac{c_3}{B}\\
& \leq \frac{r \sqrt{c\, C_2}\, \|\ushort{\delta}(f)\|_1}{\sqrt{L_N}}+ \sum_{j=1}^r \big|\E_\nu\big[F_N^{(j)}\big]\big|+ \frac{c_3}{B}.
\end{align*}
By assumption, we have, for each $j=1,\ldots,r$, $\lim_{N\to\infty} \E_\nu\big[F_N^{(j)}\big]=0$ by the central limit theorem.
Therefore we obtain
\[
\limsup_{N\to\infty}\E_\nu\left[
\sup_{1\leq j\leq r}
\int_{-B}^B \rho_j(v) \big(\dd\mathcal{A}_{N,\omega}(v)-\dd G_{0,\sigma_f^2}(v)\big)\right]\leq \frac{c_3}{B}.
\]
It now follows from \eqref{eq:A}, \eqref{eq:1} and \eqref{eq:2} we have
\[
0\leq \limsup_{N\to\infty} \E_\nu\big[\distk\big(\mathcal{A}_{N,\cdot},G_{0,\sigma_f^2}\big)\big]
\leq 2\epsilon+\frac{c_1+c_2+c_3}{B}.
\]
We now let $\epsilon$ tend to zero, then $B$ to infinity. Therefore we obtain \eqref{eq:esp-tend-0}.

{\em Second step}.  We are going to estimate the variance of $\distk\big(\mathcal{A}_{N,\cdot},G_{0,\sigma_f^2}\big)$. We want to apply \eqref{eq:momentb} with $p=1$ to the following function :
\[
F_N(\omega)=\sup_{\rho\in\mathscr{L}_0} \frac{1}{L_N}\, \mathlarger{\sum_{n=1}^N}\,
\frac{1}{n} \left( \rho\left(\frac{S_nf(\omega)}{(2n+1)^{d/2}} \right)-\int \rho \dd G_{0,\sigma_f^2}\right),
\]
since $\distk\big(\mathcal{A}_{N,\omega},G_{0,\sigma_f^2}\big)=F_N(\omega)$.
To this end, define for each $\rho\in\mathscr{L}_0$ the function
\[
F_N^{(\rho)}(\omega)=\frac{1}{L_N}\sum_{n=1}^N
\frac{1}{n} \left( \rho\left(\frac{S_nf(\omega)}{(2n+1)^{d/2}} \right)-\int \rho \dd G_{0,\sigma_f^2}\right).
\]
Let $z\in\Zd$ and $\omega,\tilde{\omega}\in\Omega$ such that $\omega_y\neq \tilde{\omega}_y$ for all $y\neq z$. We have
\[
F_N^{(\rho)}(\omega) \leq \boldsymbol{v}_z^{(N)} + F_N^{(\rho)}(\tilde{\omega})
\]
where $\boldsymbol{v}_z^{(N)}$ is defined in \eqref{def-vzn}. Now take the supremum over $\rho$ on both
sides to get
\[
F_N(\omega) \leq \boldsymbol{v}_z^{(N)} + F_N(\tilde{\omega}).
\]
The same inequality holds upon interchanging $\omega$ and $\tilde{\omega}$, hence
\[
|F_N(\omega)-F_N(\tilde{\omega})|\leq  \boldsymbol{v}_z^{(N)},
\] 
therefore
\[
\delta_z(F_N)\leq  \boldsymbol{v}_z^{(N)}.
\]
Proceeding as above, we end up with
\[
\|\ushort{\delta}(F_N)\|_2^2 \leq \frac{c\|\ushort{\delta}(f)\|_1^2}{L_N}
\]
where $c'>0$ does not depend on $N$.
Since $\nu$ satisfies $\mcb{2,C_2}$ we have
\[
\E_\nu\left[ \left(\distk\big(\mathcal{A}_{N,\cdot},G_{0,\sigma_f^2}\big)
-\E_\nu\left[\distk\big(\mathcal{A}_{N,\cdot},G_{0,\sigma_f^2}\big)\right]\right)^2\right]
\leq \frac{c'C_2\,\|\ushort{\delta}(f)\|_1^2}{L_N}.
\]
Fix $0<\delta<1$ and let $N_k=e^{k^{1+\delta}}$. From the previous inequality we get at once
\[
\sum_{k} \E_\nu\left[ \left(\distk\big(\mathcal{A}_{N_k,\cdot},G_{0,\sigma_f^2}\big)
-\E_\nu\left[\distk\big(\mathcal{A}_{N_k,\cdot},G_{0,\sigma_f^2}\big)\right]\right)^2\right]
<\infty.
\]
It follows from Beppo Levi's theorem that for $\nu$-almost every $\omega$
\begin{equation}\label{eq:d-Ed}
\lim_{k\to\infty} \Big(
\distk\big(\mathcal{A}_{N_k,\omega},G_{0,\sigma_f^2}\big)
-\E_\nu\left[\distk\big(\mathcal{A}_{N_k,\omega},G_{0,\sigma_f^2}\big)\right]
\Big)=0.
\end{equation}
By \eqref{eq:esp-tend-0}, the theorem will be proved if we can show that 
$N_k <N\leq N_{k+1}$ implies that
\begin{equation}\label{eq:final}
\big|\distk\big(\mathcal{A}_{N,\omega},G_{0,\sigma_f^2}\big)
-
\distk\big(\mathcal{A}_{N_k,\omega},G_{0,\sigma_f^2}\big)\big|\xrightarrow[]{k\to\infty} 0
\end{equation}
for $\nu$-almost every $\omega$. Indeed, if $N_k <N\leq N_{k+1}$, one has
\begin{align*}
\MoveEqLeft \big|\distk\big(\mathcal{A}_{N,\omega},G_{0,\sigma_f^2}\big)
-
\distk\big(\mathcal{A}_{N_k,\omega},G_{0,\sigma_f^2}\big)\big|\\
& \leq
\frac{L_N-L_{N_k}}{L_N}\, \distk\big(\mathcal{A}_{N_k,\omega},G_{0,\sigma_f^2}\big)\\
&  \qquad\qquad+
\sup_{\rho\in\mathscr{L}_0} \frac{1}{L_N} \sum_{n=N_k+1}^N \frac{1}{n}
\left( \rho\left(\frac{S_nf(\omega)}{(2n+1)^{d/2}} \right)-\int \rho \dd G_{0,\sigma_f^2}\right).
\end{align*}
The first term in the r.h.s goes to zero by \eqref{eq:d-Ed}. We handle the second one. We have
\begin{align*}
\MoveEqLeft[6] \left|\sup_{\rho\in\mathscr{L}_0} \frac{1}{L_N} \sum_{n=N_k+1}^N \frac{1}{n}
\left( \rho\left(\frac{S_nf(\omega)}{(2n+1)^{d/2}} \right)-\int \rho \dd G_{0,\sigma_f^2}\right)\right|\\
& \leq 
\frac{1}{L_N} \sum_{n=N_k+1}^N \frac{1}{n}
\left( \frac{\left|S_nf(\omega) \right|}{(2n+1)^{d/2}}+\int |v| \dd G_{0,\sigma_f^2}(v)\right)\\
& \leq 
\frac{1}{L_{N_k}} \sum_{n=N_k+1}^{N_{k+1}} \frac{1}{n}
\left( \frac{\left|S_nf(\omega) \right|}{(2n+1)^{d/2}}+\int |v| \dd G_{0,\sigma_f^2}(v)\right).
\end{align*}
It follows easily from our choice of $(N_k)$ that 
\[
\lim_{k\to\infty} \frac{1}{L_{N_k}} \sum_{n=N_k+1}^{N_{k+1}} \frac{1}{n}\, \int |v| \dd G_{0,\sigma_f^2}(v)=0.
\]
It remains to prove the almost-sure convergence to zero of the sequence $(U_k)$ defined by
\[
U_k=\frac{1}{L_{N_k}} \sum_{n=N_k+1}^{N_{k+1}} 
\frac{\left|S_nf(\omega) \right|}{n(2n+1)^{d/2}}.
\]
For this purpose we estimate the expectation of the square of $U_k$. Using Cauchy-Schwarz inequality and 
\eqref{eq:est-variance} we get
\begin{align*}
\E_\nu[U_k^2]  &\leq  \frac{1}{L_{N_k}^2} \sum_{n_1,n_2=N_k+1}^{N_{k+1}} 
\frac{\left(\E_\nu\left[\left(S_{n_1}f(\omega) \right)^2\right]\right)^{\frac{1}{2}}}{n_1(2n_1+1)^{d/2}}
\,
\frac{\left(\E_\nu\left[\left(S_{n_2}f(\omega) \right)^2\right]\right)^{\frac{1}{2}}}{n_2(2n_2+1)^{d/2}}\\
& \leq \frac{(\ln N_{k+1}-\ln N_k+\Oun)^2}{L_{N_k}^2}\leq \frac{\Oun}{k^2}.
\end{align*}
It follows that $\E_\nu[U_k^2]$ is summable in $k$ and by Beppo Levi's theorem we have that $U_k$ goes to zero
almost surely. Therefore we have proved \eqref{eq:final}, which finishes the proof the theorem.



\begin{thebibliography}{99}

\bibitem{ACRV}
M. Abadi, J.-R.  Chazottes, F. Redig, E. Verbitskiy.
Exponential distribution for the occurrence of rare patterns in Gibbsian random fields. 
Comm. Math. Phys. {\bf 246} (2004), no. 2, 269--294. 

\bibitem{bissacot}
Rodrigo Bissacot, Eric Ossami Endo, Aernout C. D. van Enter, Arnaud Le Ny.
Entropic repulsion and lack of the $g$-measure property for Dyson models.
Preprint, 2017.


\bibitem{bg}
S.G. Bobkov and F. G\"otze.
Exponential integrability and transportation cost related to logarithmic Sobolev inequalities.
J. Funct. Anal. {\bf 163} (1999), no. 1, 1--28. 

\bibitem{blm}
S. Boucheron, G. Lugosi, P. Massart.
Concentration inequalities.
A nonasymptotic theory of independence. Oxford University Press, 2013.

\bibitem{chatterjee1}
S. Chatterjee.
Stein's method for concentration inequalities.
Probab. Theory Related Fields {\bf 138} (2007), no. 1-2, 305--321. 

\bibitem{chatterjee-dey}
S. Chatterjee, P. S.  Dey.
Applications of Stein's method for concentration inequalities. 
Ann. Probab. {\bf 38} (2010), no. 6, 2443--2485. 

\bibitem{CCKR}
J.-R. Chazottes, P. Collet, C. K\"ulske, F. Redig.
Concentration inequalities for random fields via coupling.
Probab. Theory Related Fields  {\bf 137}  (2007),  no. 1-2, 201--225.

\bibitem{CR}
J.-R. Chazottes, F. Redig.
Occurrence, repetition and matching of patterns in the low-temperature Ising model. 
J. Stat. Phys. {\bf 121} (2005), no. 3-4, 579--605. 

\bibitem{jrf}
J.-R. Chazottes, F. Redig.
Concentration inequalities for Markov processes via coupling. 
Electron. J. Probab. {\bf 14} (2009), no. 40, 1162--1180.

\bibitem{dedecker}
J. Dedecker.
Exponential inequalities and functional central limit theorems for random fields.
ESAIM Probab. Statist. {\bf 5} (2001), 77--104.

\bibitem{dz}
A. Dembo, O. Zeitouni.
{\em Large Deviations Techniques and Applications}.
Springer, 2009.


\bibitem{dp}
D.P. Dubhashi, A. Panconesi.
{\em Concentration of Measure for the Analysis of Randomized Algorithms.}
Cambridge University Press, 2009.

\bibitem{ellis}
R. Ellis.
{\em Entropy, large deviations, and statistical mechanics.}
Classics in Mathematics. Springer-Verlag, Berlin, 2006.

\bibitem{machkouri}
M. El Machkouri. Th\'eor\`emes limites pour les champs et les suites stationnaires de variables
al\'eatoires r\'eelles. Th\`ese de doctorat de l'Universit\'e de Rouen, 2002.

\bibitem{EKW}
A. Eizenberg, Y. Kifer, B. Weiss.
Large deviations for $\mathds{Z}^d$-actions. 
Comm. Math. Phys. {\bf 164} (1994), no. 3, 433--454. 

\bibitem{vH}
R. van Handel.
Probability in High Dimension. Lecture Notes (2014) 259 pp. Available at
\url{https://www.princeton.edu/~rvan/ORF570.pdf}

\bibitem{st}
S. Gallo, D. Takahashi.
Attractive regular stochastic chains: perfect simulation and phase transition. 
Ergodic Theory Dynam. Systems {\bf 34} (2014), no. 5, 1567--1586. 

\bibitem{Geo}
H.-O. Georgii.
{\em Gibbs Measures and Phase Transitions}.
Second edition. Walter de Gruyter, 2011.

\bibitem{gross}
L. Gross.
Decay of correlations in classical lattice models at high temperature.
Comm. Math. Phys. 68 (1979), no. 1, 9--27. 

\bibitem{keller}
G. Keller.
{\em Equilibrium States in Ergodic Theory}. London Mathematical Society Student Texts vol. 42, 1998.

\bibitem{kieffer}
J. C. Kieffer.
A generalized Shannon-McMillan theorem for the action of an amenable group on a probability space.
Ann. Probability {\bf 3} (1975), no. 6, 1031--1037. 

\bibitem{kr}
A. Kontorovich, M. Raginsky.
Concentration of measure without independence: a unified approach via the martingale method. Preprint (2016):
\url{http://arxiv.org/abs/1602.00721}.

\bibitem{KT}
A. N. Kolmogorov, V. M. Tihomirov.
$\varepsilon$-entropy and $\varepsilon$-capacity of sets in functional space.
Amer. Math. Soc. Transl. (2) {\bf 17} 1961 277--364. 

\bibitem{Kul}
C. K\"ulske.
Concentration inequalities for functions of Gibbs fields
with applications to diffraction and random Gibbs measures.
Comm. Math. Phys. {\bf 239}, 29--51 (2003).

\bibitem{kunsch}
H. K\"unsch.
Decay of correlations under Dobrushin's uniqueness condition and its applications.
Communications in Mathematical Physics {\bf 84} (1982), no. 2, 207--222. 

\bibitem{ledoux}
M. Ledoux.
{\em The concentration of measure phenomenon}, Mathematical Surveys and Monographs {\bf 89}.
American Mathematical Society, Providence R.I., 2001.


\bibitem{CoyWu}
B. McCoy, T. T.  Wu.
{\em The two-dimensional Ising model}. Harvard University Press, Cambridge, MA, 1973.

\bibitem{MRSvM}
C. Maes, F. Redig, S. Shlosman, A. Van Moffaert.
Percolation, path large deviations and weakly Gibbs states. 
Comm. Math. Phys. {\bf 209} (2000), no. 2, 517--545. 

\bibitem{martinlof}
A. Martin-L\"of.
Mixing properties, differentiability of the free energy and
the central limit theorem for a pure phase in the Ising model at low temperature.
Comm. Math. Phys. {\bf 32} (1973), 75--92.

\bibitem{marton0}
K. Marton.
Bounding $\db$-distance by informational divergence: a method to prove measure concentration. 
Ann. Probab. {\bf 24} (1996), no. 2, 857--866. 

\bibitem{marton1}
K. Marton.
Measure concentration and strong mixing.
Studia Sci. Math. Hungar. {\bf 40} (2003), no. 1-2, 95--113.

\bibitem{marton2}
K. Marton.
Measure concentration for Euclidean distance in the case of dependent random variables.
Ann. Probab. {\bf 32} (2004), no. 3B, 2526--2544.

\bibitem{RS}
L. R\"uschendorf, T. Sei.
On optimal stationary couplings between stationary processes. 
Electron. J. Probab. {\bf 17} (2012), no. 17, 20 pp. 

\bibitem{salassokal}
J. Salas, A. D. Sokal.
Absence of phase transition for antiferromagnetic Potts models via the Dobrushin uniqueness theorem. 
J. Statist. Phys. {\bf 86} (1997), no. 3-4, 551--579. 

\bibitem{schonmann}
R. Schonmann.
Second order large deviation estimates for ferromagnetic systems in the phase coexistence region.
Comm. Math. Phys. {\bf 112} (1987), no. 3, 409--422. 


\bibitem{tempelman}
A. Tempelman.
{\em Ergodic theorems for group actions.}
Informational and thermodynamical aspects. Mathematics and its Applications, 78. Kluwer Academic Publishers Group, Dordrecht, 1992.

\end{thebibliography}
\end{document}